\documentclass[10pt]{siamltex}
\usepackage{graphics,color}
\usepackage{amsfonts}
\usepackage{amssymb}
\usepackage{amsmath}
\usepackage[shortlabels]{enumitem}
\usepackage{url}
\usepackage{bm}
\usepackage[footnotesize,FIGBOTCAP,TABTOPCAP]{subfigure}
\usepackage[footnotesize]{caption2} 
\usepackage{graphics}
\usepackage{graphicx}
\usepackage{mathabx}
\usepackage{mathrsfs}
\usepackage{cite}
\usepackage[colorlinks=true, bookmarksopen,
            pdfauthor={Hoang Tran},
            pdfcreator={pdftex},
            pdfsubject={algorithms},
            linkcolor={blue},
            anchorcolor={black},
            citecolor={red},
            filecolor={magenta},
            menucolor={black},
            plainpages=false,pdfpagelabels,
            urlcolor={blue}]{hyperref}

\newtheorem{assumption}[theorem]{Assumption}

\newcounter{saveeqn}%

\newtheorem{defi}[theorem]{Definition}
\newtheorem{remark}[theorem]{Remark}

\renewcommand{\theequation}{\thesection.\arabic{equation}}
\numberwithin{equation}{section}

\title{Convergence analysis for a nonlocal gradient descent method via directional Gaussian smoothing\thanks{This manuscript has been authored by UT-Battelle, LLC, under contract DE-AC05-00OR22725
with the US Department of Energy (DOE). The US government retains and the publisher, by accepting the
article for publication, acknowledges that the US government retains a nonexclusive, paid-up, irrevocable,
worldwide license to publish or reproduce the published form of this manuscript, or allow others to do so, for
US government purposes. DOE will provide public access to these results of federally sponsored research in
accordance with the DOE Public Access Plan.}}
\author{Hoang Tran\thanks{Computer Science and Mathematics Division, Oak Ridge National Laboratory,
    Oak Ridge, TN 37831, USA, Email: \textup{\nocorr
      \texttt{tranha@ornl.gov}}}
\and 
Qiang Du\thanks{Department of Applied Physics and Applied Mathematics and Data Science Institute, Columbia University, New York, NY 10027, USA, Email: \textup{\nocorr
      \texttt{qd2125@columbia.edu}}
}
\and 
Guannan Zhang\thanks{Computer Science and Mathematics Division, Oak Ridge National Laboratory,
    Oak Ridge, TN 37831, USA, Email: \textup{\nocorr
      \texttt{zhangg@ornl.gov}. (Corresponding Author).}
}} 

\begin{document}
\maketitle
\renewcommand{\thefootnote}{\fnsymbol{footnote}}

\begin{abstract}
We analyze the convergence of a nonlocal gradient descent method for minimizing a class of high-dimensional non-convex functions, where a directional Gaussian smoothing (DGS) is proposed to define the nonlocal gradient (also referred to as the DGS gradient).
The method was first proposed in \cite{DBLP:conf/uai/ZhangTLZ21}, in which multiple numerical experiments showed that replacing the traditional local gradient with the DGS gradient can help the optimizers escape local minima more easily and significantly improve their performance. However, a rigorous theory for the efficiency of the method on nonconvex landscape is lacking. In this work, we investigate the scenario where the objective function is composed of a convex function, perturbed by a oscillating noise. We provide a convergence theory under which the iterates exponentially converge to a {tightened} neighborhood of the solution, whose size is characterized by the noise wavelength. {We also establish a correlation between the optimal values of the Gaussian smoothing radius and the noise wavelength, thus justify the advantage of using moderate or large smoothing radius with the method.} Furthermore, if the noise level decays to zero when approaching global minimum, we prove that DGS-based optimization converges to the exact global minimum with linear rates, similarly to standard gradient-based method in optimizing convex functions. {Several numerical experiments are provided to confirm our theory and illustrate the superiority of the approach over those based on the local gradient.} 
\end{abstract}

\section{Introduction}
We are interested in solving the optimization problem
\begin{equation}\label{e1}
    \min_{\bm x \in \mathbb{R}^d} \phi(\bm x),
\end{equation}
where $\bm x = (x_1, \ldots, x_d) \in \mathbb{R}^d$ consists of $d$ parameters, and $\phi: \mathbb{R}^d \rightarrow \mathbb{R}$ is a high-dimensional, well-behaved objective function.
However, we assume that neither $\phi(\bm x)$ nor the gradient $\nabla \phi(\bm x)$ is {readily} available, and that the function $\phi(\bm x)$ is only accessible via a noisy approximate $F(\bm x) = \phi(\bm x) + \epsilon(\bm x)$, where $\epsilon$ represents a highly oscillating noise that perturbs the true objective function. The scenario can be found in 
{various}
applications, for instance, machine learning \cite{Real_ICML17,Houthooft18,10.1145/3205455.3205517}, model calibration \cite{Radaideh-NimB} and experimental design \cite{10.1007/978-3-319-99259-4_24,Molesky-nano}, where the loss landscape of objective functions is non-convex and highly 
{rugged and complex}, with its 
{main geometric features}
being concealed under the small-scale, deceptive fluctuations. In such cases, conventional gradient-based algorithms are not 
{effective} because $F$ has many local minima that would trap the optimizers.

One promising strategy for optimizing such noisy functions is using Gaussian smoothing (GS) \cite{FKM05,NesterovSpokoiny15}. GS first smooths 
the loss function with Gaussian convolution and then use the gradient of the smoothed function to guide the optimization. Monte Carlo (MC) sampling is often used to estimate the Gaussian convolution {\cite{SHCS17}}, but this approach suffers from low accuracy and high variance because the GS requires a $d$-dimensional integration. To enhance the MC estimators, several methods have been proposed, 
such as by variance reduction \cite{MWDS18,CRSTW18}, exploiting historical data \cite{Maheswaranathan_GuidedES,Meier_OPTRL_2019}, employing active subspaces \cite{Choromanski_ES-Active}, and searching on latent low-dimensional manifolds \cite{Sener2020Learning}. 
Despite remarkable improvements, this line of works still relies on small smoothing radius to maintain MC accuracy. 
{As a result, their} main focus 
{has been}
essentially 
{on} finding more accurate estimates of the local gradient. {These existing works so far} did not address a fundamental challenge of optimization on 
{complex} non-convex landscapes, 
{namely, how to avoid} 
getting trapped at undesirable local optima. 

In \cite{DBLP:conf/uai/ZhangTLZ21}, a 
{new}
nonlocal gradient was developed for global black-box optimization, using {the} directional Gaussian smoothing (DGS) approach. We refer to such gradient {\em the DGS gradient} hereinafter. {While there have been a great deal of studies on nonlocal gradient operators motivated by many different applications \cite{doi:10.1137/1.9781611975628,doi:10.1142/S0218202512500546}, the}
key idea behind the DGS gradient 
{in the $n$-dimensional space} is to conduct 1D nonlocal explorations along $d$ {mutually} orthogonal directions in $\mathbb{R}^d$, each of which defines a nonlocal directional derivative as a 1D integral. Then, the $d$ directional derivatives are assembled to form the gradient operator. Compared with the traditional GS approach, \cite{DBLP:conf/uai/ZhangTLZ21} demonstrated that DGS method can use large smoothing radius to achieve long-range exploration, because it is much easier {and robust} to approximate the DGS gradient accurately {with limited samplings}. Hence, the DGS gradient provides better search directions than the local gradient {of highly noisy objective functions}, making it particularly suitable for the global optimization.

Yet, the study and demonstration of DGS gradient in \cite{DBLP:conf/uai/ZhangTLZ21} were mainly empirical, and what is currently lacking is a supporting theory to explain the strong performance of DGS approach, especially when using moderate and large smoothing radius. In this paper, we aim to fill this gap by providing a convergence analysis of a gradient descent scheme built on the DGS gradients. We 
{focus on the setting where the} noisy objective function is composed of a convex function perturbed by a deterministic wave-like noise. Generally, an optimizer can only converge to a neighborhood of the solution and not the exact solution in this case. However, our analysis 
show{s} that DGS gradients can significantly tighten 
{such a} neighborhood, because they can efficiently smooth out the small-scale fluctuations {present in} the noise. In particular, we prove {that,} if the smoothing radius is chosen to be approximately the wavelength of the noise functions, the size of the neighborhood of convergence 
also scales linearly with this wavelength. This is optimal in the sense that without any assumption on the evolution of the noise function 
{over} the interval of one wavelength, it is not possible to guarantee the optimizer can break this {so-called} \textit{wavelength limit}. The derived neighborhood of convergence is also much 
smaller compared to the cases where 
small radius is chosen (or local gradients are relied on for search direction), therefore justifying the advantage of using moderate or 
{large} radius for {the} Gaussian smoothing approach. Further, {if the noise perturbation} diminishes quadratically to $0$ when approaching the global minimum, we prove our DGS scheme converges linearly to the \textit{exact} solution with a decreasing sequence of smoothing radius. The result shows that progressively decreasing the smoothing radius after reaching the neighborhood of convergence can help to break 
{the aforementioned wavelength}
limit. Lastly, several numerical experiments are provided to confirm and illustrate our theoretical analysis. 

\textbf{Literature review.} Our method fits into the category of optimization under the presence of noise. These noises can arise in multiple different scenarios. Deterministic noise appears when the objective functions are outputted from the iterative solution of linear systems of equations, where some tolerance thresholds are applied for stopping criteria, or from the adaptation of the grid discretization in partial differential equation solvers. Stochastic noise, on the other hand, appears when the objective function is stochastic, e.g., there are random fluctuations or measurement errors in the data which cannot be replicated. The development and analysis of numerical optimization under the presence of noise are extensive. In \cite{doi:10.1137/110830629}, the convergence rate of gradient descent method was analyzed under a generic sequence of upper bound of gradient noise. Finite-difference quasi-Newton approach for noisy functions was developed in \cite{doi:10.1137/18M1177718}, where the differencing intervals were adjusted according to the level of noise. In \cite{doi:10.1137/19M1240794}, a noise tolerant version of the BFGS method was proposed and analyzed, followed by extensions in \cite{doi:10.1137/20M1373190} to make the method more robust and efficient in practice. Trust region methods with noise was developed and analyzed in \cite{2022arXiv220100973S}. Recently, \cite{doi:10.1080/10556788.2022.2121832} re-examined the finite difference approach and investigated its empirical performance relative to more advanced techniques, and \cite{BCCS19} presented theoretical and numerical comparisons of multiple methods, including finite differences, linear interpolation, Gaussian smoothing, in approximating gradients of noisy functions and convergence rates. In those works, proposed optimization methods, by introducing strategies to extract and/or take advantage of the noise estimation \cite{Mor2011EstimatingCN}, enjoys convergence to a neighborhood of stationary points, whose size is determined by the size of noise. Our present result is similar in this aspect; however, our method is the first to exploit and show convergence up to the \textit{noise wavelengths}, unlike most of the aforementioned works where the noise is assumed to be uniformly bounded (in expectation in case of stochastic noise). Another related approach is implicit filtering \cite{doi:10.1137/S1052623499354096,doi:10.1137/1.9781611971903}, which is designed for the case where noise diminishes as the iterates approach the solution, hence convergence to global minimum is guaranteed. Implicit filtering is based on gradient projection algorithm and applies to noisy optimization problems with bound constraints.   

The noise in optimization can also arise in the scenario that the optimization algorithm itself is a probabilistic process, regardless the objective function is deterministic or stochastic. As an example, the objective function is defined as an expectation across a distribution, which optimization algorithm can only tackle indirectly via a set of sample points. In this case, one can control the accuracy of gradient approximation by the sample size, and adaptive sampling strategies have been introduced with global convergence being guaranteed \cite{doi:10.1137/17M1154679, MWDS18}. Studies of general classes of probabilistic models within trust region and line search frameworks can be found in \cite{doi:10.1137/130915984, CartisScheinberg_MathProgram}  

Perhaps the simplest method that used GS to assist the optimization of noisy or black-box functions is two-point approach. This type of methods randomly generates the next search direction and then estimates the directional derivative with GS for the updates. A theoretical analysis of two-point schemes was presented in the seminal paper \cite{NesterovSpokoiny15}, sharpened and extended in \cite{DBLP:journals/tit/DuchiJWW15,doi:10.1137/120880811,10.5555/3122009.3153008} for non-convex, non-smooth loss functions. A more involved approach is to find the search direction by accumulating multiple directional estimates by two-point schemes. The GS-based evolutionary strategy (ES) \cite{SHCS17,Mania2018SimpleRS, Maheswaranathan_GuidedES,CRSTW18,Choromanski_ES-Active} can be assigned to this category. However, the current works only used GS as a vehicle to approximate \textit{local} gradients, and our paper is the first to exploit and analyze the benefit of moderate and large smoothing radius for \textit{nonlocal} exploration and more efficient optimization.    

Finally, for recent and more thorough reviews on optimization under the presence of noise and stochastic optimization, we refer to \cite{9194022,Larson_et_al_19}.

\textbf{Paper organization.} The paper is organized as follows. In Section \ref{sec:setting}, we present the DGS algorithm, as well as much of the standing notation and conventions used throughout the paper. This is followed by Section \ref{sec:statement}, where we discuss mathematical models to capture the small scale fluctuations of the noise. We establish upper bounds of the DGS gradient on these noise models, which form a critical component in our {subsequent}  analysis. The main results on the linear convergence of DGS method are presented in Section \ref{sec:analysis}. We show that the neighborhood of convergence can be significantly tightened with an optimal choice of {the} smoothing radius, depending on the frequency of the noise functions. A linear convergence to the exact solutions is guaranteed with quadratically diminishing noise. {Section} \ref{sec:experiment} contains the numerical experiments illustrating and confirming our theoretical results. Finally, we provide concluding remarks and outlooks in Section \ref{sec:conclusion}.

\section{The DGS algorithm}\label{sec:setting}
Recall that we are interested in minimizing
\begin{equation*}
    \min_{\bm x \in \mathbb{R}^d} \phi(\bm x),
\end{equation*}
where we can only access a noisy approximation of {the objective function} $\phi$, i.e., $F = \phi + \epsilon$. Using {the} GS technique, we aim to design an efficient smoothing process for $F$ with some Gaussian kernels, with the goal of reducing the noise in $F$ and obtaining good descent directions to guide the optimizer. Throughout this paper, we make some standard assumptions on 
$\phi$. 

\begin{assumption}[Lipschitz continuity of the gradient of $\phi$]
\label{assump:lipschitz_cont}
The function $\phi$ is continuously differentiable, and the gradient of $\phi =\phi(\bm{x})$ is $L$-Lipschitz continuous for all $\bm{x} \in \mathbb{R}^d$, i.e.,  
\begin{align*}
    \|\nabla \phi(\bm x) - \nabla \phi(\bm y) \| \le L \|\bm x - \bm y\|, \ \ \forall \bm x, \bm y\in \mathbb{R}^d. 
\end{align*}
\end{assumption}

\begin{assumption}[Strongly convexity of $\phi$]
\label{assump:strongly_convex}
The function $\phi$ is strongly convex,{ i.e.,}
\begin{align*}
   \phi(\bm y) \ge \phi(\bm x) + \langle \nabla \phi(\bm x),\bm y - \bm x \rangle + \frac{\tau}{2} \|\bm y - \bm x\|^2, \ \ \forall \bm x, \bm y\in \mathbb{R}^d. 
\end{align*}
Here, $\tau\ge 0$ is the convexity parameter. 
\end{assumption}

The standard GS \cite{NesterovSpokoiny15,SHCS17} applies a global smoothing to $F$, where the smoothed loss is defined by
$
      F_{\sigma}(\bm x) = \mathbb{E}_{\bm u \sim \mathcal{N}(0, \mathbf{I}_d)} \left[F(\bm x + \sigma \bm u) \right], 
$
with $\mathcal{N}(0, \mathbf{I}_d)$ being the $d$-dimensional standard Gaussian distribution, and $\sigma > 0$ the smoothing radius. 
{The choice of $\sigma$ signifies the range of the nonlocal smoothing effect.} Although $\nabla F_\sigma(\bm x)$ 
{could be} a great candidate for search directions in multimodal landscapes, this gradient is not easy to realize because it involves a $d$-dimensional integral, 
\begin{equation}\label{e40}
\begin{aligned}
    \nabla F_{\sigma}(\bm x) & = 
    \frac{1}{\sigma}\mathbb{E}_{\bm u \sim \mathcal{N}(0, \mathbf{I}_d)} \left[F(\bm x + \sigma \bm u)\, \bm u\right] {,}
\end{aligned}
\end{equation}
{which is hard to estimate with high accuracy in practice.}
%
%
Most existing works used MC to estimate $\nabla F_\sigma(\bm x)$ in Eq.~\eqref{e40}, 
{which} essentially limited $\sigma$ 
to small values. Since $\lim_{\sigma \rightarrow 0} \nabla F_{\sigma}(\bm x) = \nabla F(\bm x)$, GS could   
just {be viewed as}
another way to estimate the standard gradient when 
{the latter} cannot be accessed directly.  

\subsection{The DGS gradient}\label{sec:grad}
The work \cite{DBLP:conf/uai/ZhangTLZ21} represents a recent effort to utilize GS to derive search directions with nonlocal exploration capability. Rather than smoothing $F$ globally in the function domain, \cite{DBLP:conf/uai/ZhangTLZ21} applies the GS directionally along $d$ orthogonal directions, each of which now defines a nonlocal directional derivative as a 1D integral. To form the DGS gradient thus requires computing $d$ 1D integrals instead of one $d$-dimensional integral, which is much cheaper and can be realized with multiple available tools, such as Gaussian-Hermite quadrature. 

In particular, let us first define a 1D cross section of $F(\bm x)$ as
\begin{equation*}
G(y \,| \,{\bm x, \bm \xi}) = F(\bm x + y\, \bm \xi), \;\; y \in \mathbb{R},
\end{equation*}
where $\bm x$ is the current state of $F(\bm x)$ and $\bm \xi$ is a unit vector in $\mathbb{R}^d$. Note that $\bm x$ and $\bm \xi$ can be viewed as parameters of the function $G$. 
{Next, we define the standard 1D Gaussian kernel function 
\begin{align}
\label{def:Gauss_kernel}
{g_{\sigma}}(v) =  \dfrac{1}{\sigma\sqrt{2\pi}}\exp\left(-\dfrac{v^2}{2\sigma^2}\right) , \ \forall v \in \mathbb{R}. 
\end{align}
We then} define the Gaussian smoothing of $G(y)$, denoted by $G_\sigma(y)$, by
\begin{equation}
\label{eq10}
\begin{aligned}
    G_{\sigma}(y \,| \,{\bm x, \bm \xi}) & = 
   { \int_{\mathbb{R}} G(y + 
    v\, |\, \bm x, \bm \xi)\, {g_{\sigma}}(v) 
    \, dv = } \frac{1}{\sqrt{2\pi}} \int_{\mathbb{R}} G(y + \sigma
    v\, |\, \bm x, \bm \xi)\, {\rm e}^{-\frac{v^2}{2}}
    \, dv\\
    & = \mathbb{E}_{v \sim \mathcal{N}(0, 1)} \left[G(y + \sigma v\, |\, \bm x, \bm \xi) \right],
\end{aligned}
\end{equation}
which is also the Gaussian smoothing of $F(\bm x)$ along the direction $\bm \xi$ in the neighbourhood of $\bm x$. 
The derivative of $G_{\sigma}(y|\bm x,\bm \xi)$ at $y = 0$ can be represented by a 1D expectation
\begin{equation}\label{e4}
    \mathscr{D}[G_{\sigma}(0 \,|\, \bm x, \bm \xi)] 
     = \frac{1}{\sigma}\,\mathbb{E}_{v \sim \mathcal{N}(0,1)} \left[G(\sigma v \, | \, \bm x, \bm \xi)\, v\right],
\end{equation}
where $\mathscr{D}[\cdot]$ denotes the differential operator. As Eq.~\eqref{e4} is a 1D integral, it is easier to conduct long-range exploration with large values of the smoothing radius $\sigma$.

To form the DGS gradient, given a matrix $\bm \Xi := (\bm \xi_1, \ldots, \bm \xi_d)$ consisting of 
$d$ orthonormal vectors, we define $d$ directional derivatives like Eq.~\eqref{e4} and assemble them as 
%
%
\begin{equation}\label{dev_smooth_func}
    {\nabla}_{\sigma, \bm \Xi}[F](\bm x) = \bm \Xi^{\top}
\begin{bmatrix}
{\mathscr{D}}\left[G_{\sigma}(0 \, |\, \bm x, \bm \xi_1)\right]\\
\vdots \\
{\mathscr{D}}\left[G_{\sigma}(0\, |\, \bm x, \bm \xi_d) \right]
\end{bmatrix}.
\end{equation}
%




\subsection{The Gauss-Hermite quadrature estimator}\label{sec:ada_DGS-ES}
It remains to find an accurate estimator for the DGS gradient in Eq. \eqref{dev_smooth_func}.
%
Since each component of ${\nabla}_{\sigma, \bm \Xi}[F](\bm x)$ is a 1D integral, the GH quadrature rule \cite{Handbook} is perfectly suitable for approximating the integrals with high accuracy. Specifically, the GH rule can be directly used to obtain the following estimator for $\mathscr{D}[G_{\sigma}(0 \,|\, \bm x, \bm \xi)]$, i.e.,
\begin{equation}\label{e8}
\begin{aligned}
  \widetilde{\mathscr{D}}^M[G_\sigma(0 \, | \, \bm x, \bm \xi)] 
     =  \frac{1}{\sqrt{\pi}\sigma} \sum_{m = 1}^M w_m \,F(\bm x + \sqrt{2}\sigma v_m \bm \xi)\sqrt{2}v_m, 
\end{aligned}
\end{equation}
where $\{v_m\}_{m=1}^M$ are the roots of the $M$-th order Hermite polynomial
 and $\{w_m\}_{m=1}^M$ are quadrature weights.  
It was theoretically proved in \cite{Handbook} that the error of the estimator in Eq.~\eqref{e8} is 
\begin{align*}
\hspace{-0.1cm}\Big|(\widetilde{\mathscr{D}}^M- \mathscr{D})[G_\sigma] \Big| \sim \frac{M\,!\sqrt{\pi}}{2^M(2M)\,!}, 
\end{align*}
where $M!$ is the factorial of $M$. In comparison, the error of an MC estimator is on the order of $1/\sqrt{M}$.
Applying the GH quadrature to each component of ${\nabla}_{\sigma, \bm \Xi}[F](\bm x)$ in Eq.~\eqref{dev_smooth_func}, we obtain the final estimator for the DGS gradient: 
%
\begin{equation}\label{e5}
    \widetilde{\nabla}^M_{\sigma, \bm \Xi}[F](\bm x) := \bm \Xi^{\top}
\begin{bmatrix}
\widetilde{\mathscr{D}}^M\left[G_{\sigma}(0 \, |\, \bm x, \bm \xi_1)\right]\\
\vdots \\
\widetilde{\mathscr{D}}^M\left[G_{\sigma}(0\, |\, \bm x, \bm \xi_d) \right]
\end{bmatrix},
\end{equation}
which requires a total of $M\times d$ function evaluations. 

\subsection{Gradient descent scheme with the DGS gradient} Once formed, the DGS operator \eqref{e5} can be fed into any gradient-based schemes in place of the local gradient to navigate the optimizers through noisy, oscillating landscapes. This paper is concerned with the analysis of the simple gradient descent scheme, where DGS gradient replaces the standard gradient as
 \begin{equation}
 \label{alg:GD}
\bm x_{t+1} = \bm x_t - \lambda \widetilde{\nabla}^M_{\sigma_t, \bm \Xi}[F](\bm x_t).
\end{equation}
Here, $\bm x_t$ and ${\bm x_{t+1}}$ are the candidate solutions at iterations $t$ and $t+1$, $\lambda$ is the step size and $\sigma_t$ is the smoothing radius. Note that we consider fixed step size, but allow varying smoothing radius in some scenarios studied next. 

\section{The noise models and bounds on the DGS gradient of the noise}
\label{sec:statement}
In this section, we 
discuss some {representative} models 
{of} the noise function, 
and establish useful bounds on the DGS gradients 
for each case. These bounds 
{are} critical for our algorithm analysis in the next section. 

 Define the 1D cross section of {the noise function} $\epsilon(\bm x)$ as 
\begin{align}
\label{def:eta}
\eta(y|\bm x, \bm \xi ) = \epsilon(\bm x + y\bm \xi), \  y\in \mathbb{R}. 
\end{align}
$\bm x$ and $\bm \xi$ can be viewed as parameters of the function $\eta$, and {for simplicity} we also refer to $\eta(\cdot | \bm x, \bm \xi)$ as 
$\eta$ if no confusion arises. 
{Similar to the definitions given in the} Section \ref{sec:grad}, the 1D Gaussian smoothing of $\epsilon(\bm x)$ along the direction $\bm \xi$ with {the} smoothing radius $\sigma$ is represented by
\begin{equation}
\label{def:eta_s}
\begin{aligned}
    \eta_{\sigma}(y \,| \,{\bm x, \bm \xi}) & = 
     { \int_{\mathbb{R}} \eta(y + 
    v\, |\, \bm x, \bm \xi)\, {g_{\sigma}}(v) 
    \, dv = } \frac{1}{\sqrt{2\pi}} \int_{\mathbb{R}} \eta(y + \sigma v\, |\, \bm x, \bm \xi)\, {\rm e}^{-\frac{v^2}{2}}\, dv\\
    & = \mathbb{E}_{v \sim \mathcal{N}(0, 1)} \left[\eta(y + \sigma v\, |\, \bm x, \bm \xi) \right],
\end{aligned}
\end{equation}
and we have the DGS gradient of $\epsilon$ at $\bm x$
\begin{equation}
    {\nabla}_{\sigma, \bm \Xi}[\epsilon](\bm x) = \bm \Xi^{\top}
\begin{bmatrix}
{\mathscr{D}}\left[\eta_{\sigma}(0 \, |\, \bm x, \bm \xi_1)\right]\\
\vdots \\
{\mathscr{D}}\left[\eta_{\sigma}(0\, |\, \bm x, \bm \xi_d) \right]
\end{bmatrix}.
\end{equation}
{With the Gaussian kernel function $g_{\sigma}$ given in \eqref{def:Gauss_kernel} and its derivative}
${g_{\sigma}'}(v) = -{\frac{1}{\sigma^2} {g_{\sigma}}(v)v} 
$, 
we have that  
\begin{gather}
\label{eq:8}
\begin{aligned}
\eta_\sigma(0|\bm x,\bm \xi) & = \mathbb{E}_{v \sim \mathcal{N}(0, 1)} \left[\eta( \sigma v\, |\, \bm x, \bm \xi) \right] =  \frac{1}{\sqrt{2\pi}} \int_{-\infty}^{\infty}\eta( \sigma v\, |\, \bm x, \bm \xi) \, {\rm e}^{-\frac{v^2}{2}}\, dv .
\\
&  = \frac{1}{\sigma \sqrt{ 2\pi}} \int_{-\infty}^{\infty} \eta(- v\, |\, \bm x, \bm \xi)\, {\rm e}^{-\frac{v^2}{2\sigma^2}}\, dv = ({g_{\sigma}} * \eta) (0)
\\
{\mathscr{D}}\left[\eta_{\sigma}(0 \, |\, \bm x, \bm \xi)\right] &  = - \frac{1}{\sigma^3\sqrt{ 2\pi}} \int_{-\infty}^{\infty} \eta(- v\, |\, \bm x, \bm \xi) \, {\rm e}^{-\frac{v^2}{2\sigma^2}}
\, v dv =  ({g_{\sigma}'} * \eta) (0 ).
\end{aligned}
\end{gather}
Here, $*$ is the convolution operator. 

Let $\bm x^*$ be the global minimum of $\phi$. We are concerned with two different scenarios for {the noisy} $\epsilon$. To describe noises characterized by 
fluctuations {on small scales}, {in the first model}, we define wave-like noise models, including the high-frequency bandlimited (see definition below) and the simplified, yet more explanatory periodic models. The second class of noise considered here is diminishing noise, which is for demonstrating that the DGS optimization can converge to global minimum in such case.   
\begin{enumerate}[(i)]
\item \textit{Wave-like noise:}
\begin{itemize}
\item \textit{periodic:} cross-sections of $\epsilon$ along $\bm \xi_1,\ldots, \bm \xi_d$ are periodic, 
\item \textit{high-frequency bandlimited:} the power spectra of cross-sections of $\epsilon$ along $\bm \xi_1,\ldots, \bm \xi_d$ possess a uniform positive lower bound,  
\end{itemize}
\item \textit{Diminishing noise:} $\epsilon(\bm x)$ is quadratically diminishing as $\bm x$ approaches $\bm x^*$. 
\end{enumerate}
 {As the initial attempt to carry out the analysis on the DGS scheme, it is interesting to focus on these two models in this work as the starting point, since they can be representative to noises often encountered in practice: in the first case, we may view the noise $\epsilon$ being additive, with its magnitude unrelated to $\phi$; while in the second case, the noise is multiplicative with its magnitude related to $\phi$. Indeed, one may see the latter case being fluctuations on a relative scale.}
Figure \ref{fig:noise_model} 
{offers} an illustration of $\epsilon$ in these two scenarios. We proceed to estimate ${\nabla}^M_{ \sigma, \bm \Xi}[\epsilon](\bm x)$, which is the DGS gradient of the noise $\epsilon$. Since ${\nabla}^M_{ \sigma, \bm \Xi}[\epsilon](\bm x)$ represents a perturbation to the search directions informed by true objective function $\phi$, it is desirable that this quantity to be small. 

\begin{figure}[h]
\vspace{-.1cm}
\includegraphics[width = 1.6in]{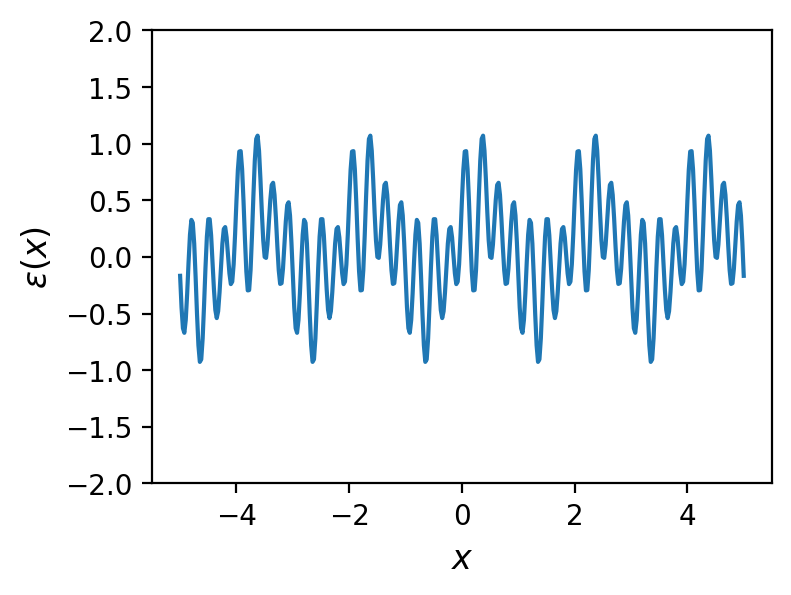}
\includegraphics[width = 1.6in ]{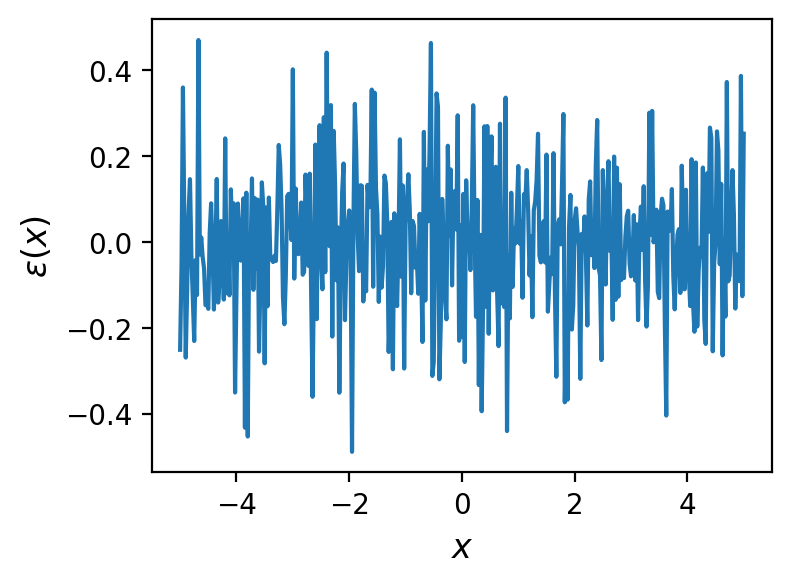}
\includegraphics[width = 1.6in]{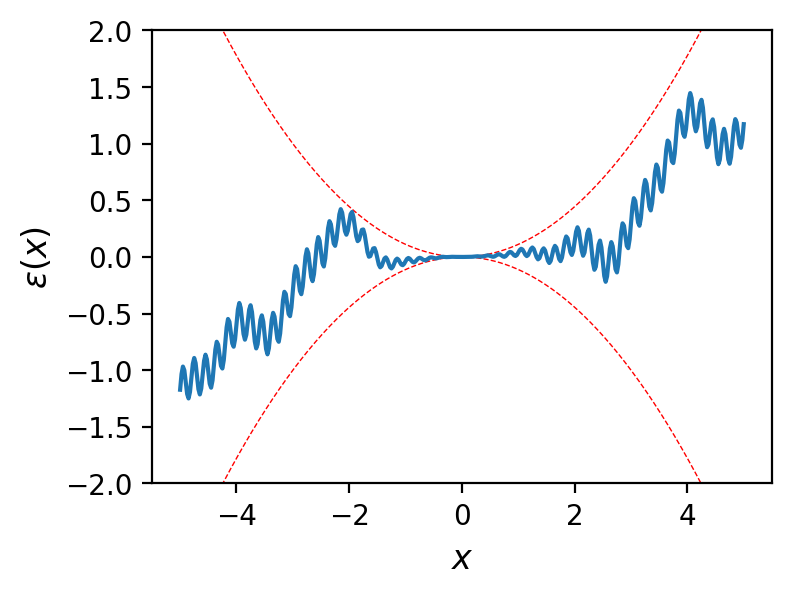}
\vspace{-.1in}
\caption{Illustrations of the noise models for $\epsilon$ considered in this study. From left to right: wave-like (periodic), wave-like (high-frequency bandlimited), and diminishing noise.}
\label{fig:noise_model}
\vspace{-.2in}
\end{figure}


\subsection{Wave-like noise}
First, we analyze the periodic noise. This simple model helps to highlight the connection between the optimal smoothing radius and the frequency of the noise functions.
\begin{proposition}[Periodic noise]
\label{prop:periodic}
Let $\epsilon$ be a real function on $\mathbb{R}^d$ 
{and its}
cross sections $\eta(\cdot|\bm x, \bm \xi_i)$ 
along $\bm \xi_1,\ldots, \bm \xi_d$  
{be} periodic functions with {the} period $1/\alpha$. Assume that for some $n\in \mathbb{N}$, $\eta(\cdot|\bm x,\bm\xi_i)$ are 
continuously differentiable {of order $n$} with uniformly bounded $n$-th derivatives, i.e.,  $|{\eta^{(n)}}(y|\bm x,\bm \xi_i)|\, \le 
{\gamma_n<\infty},\ \forall y\in \mathbb{R},\ i\in \{1,\ldots,d\}$. Then 
\begin{gather*}
\begin{aligned}
\|{\nabla}_{\sigma, \bm \Xi}[\epsilon](\bm x)\| & \le 
\begin{cases}
\displaystyle 
\frac{\gamma_1\sqrt{d}}{2} \exp(- 2\pi^2\alpha^2 \sigma^2) \left({1} + \frac{1}{2\alpha\sigma\sqrt{2\pi}}\right), & \text{if}\ n = 1,
\\
\displaystyle
\frac{ {\gamma_2}\sqrt{d}}{4\pi\alpha} \exp(- 2\pi^2\alpha^2 \sigma^2)\left(1 + \frac{1}{2}\log\left(1+\frac{1}{2\pi^2\alpha^2\sigma^2}\right)\right), &  \text{if}\ n = 2,
\\
\displaystyle
 \frac{\gamma_n \sqrt{d}}{2 (2\pi  \alpha  )^{n-1}} \exp(- 2\pi^2\alpha^2 \sigma^2) \left( 1 + \frac{1}{4\pi^2 \alpha^2\sigma^2}\right),&  \text{if}\ n > 2.
\end{cases}
\end{aligned}
\end{gather*}
\end{proposition}




\begin{proof}
Note that $\|{\nabla}_{ \sigma, \bm \Xi}[\epsilon](\bm x)\|^2 = 
\sum\limits_{i=1}^d |{\mathscr{D}}\left[\eta_{\sigma}(0 \, |\, \bm x, \bm \xi_i)\right] |^2 
$, we will first estimate 
{the components} of ${\nabla}_{ \sigma, \bm \Xi}[\epsilon](\bm x)$. It is easy to see from \eqref{eq:8} that ${\mathscr{D}}\left[\eta_{\sigma}(0 \, |\, \bm x, \bm \xi_i)\right] $ is completely determined by the odd part of $\eta$, so without loss of generality, we can assume $\eta$ is an odd function.  {For each $i$,} since $\eta(y | \bm x, \bm \xi_i)$ is a periodic function
in $y$ and 
{has} continuous bounded derivatives, 
{it} can be represented as  
\begin{align}
\label{new_eq:10}
\eta(y | \bm x, \bm \xi_i) =
\sum_{k=1}^{\infty} b_k \sin\left({2\pi\alpha k y}{}\right), 
\end{align}
where from integration by part
\begin{gather}
\label{new_eq:11}
\begin{aligned}
|b_k|  & 
=\left|
\frac{\alpha}{2}  \int_{0}^{1/\alpha} \eta(v |\bm x,  \bm \xi_i) \sin\left({2\pi\alpha k v}\right) dv
\right|\\
&
\le 
\begin{cases} \displaystyle
\dfrac{\alpha}{2(2\pi\alpha k)^n} \int\limits_{0}^{1/\alpha} |\eta^{(n)}(v |\bm x,  \bm \xi_i)\sin\left({2\pi\alpha k v}\right) |dv, \ \ \text{ if } n \text{ is even }
\\  \displaystyle
\dfrac{\alpha}{2(2\pi\alpha k)^n} \int\limits_{0}^{1/\alpha} |\eta^{(n)}(v |\bm x,  \bm \xi_i)\cos\left({2\pi\alpha k v}\right) | dv, \ \ \text{ if } n \text{ is odd }
\end{cases}
\\
& \le 
\dfrac{\gamma_n}{2(2\pi\alpha k)^n}, \quad \forall k\geq 1.
\end{aligned}
\end{gather}

From \eqref{new_eq:10}, it gives
\begin{align*} 
 {\mathscr{D}}\left[\eta_{\sigma}(0 \, |\, \bm x, \bm \xi_i)\right]  & =  {(g_{\sigma}' * \eta) (0 )}  = - \frac{1}{\sigma^3 \sqrt{ 2\pi}} \int_{-\infty}^{\infty} \eta(- v\, |\, \bm x, \bm \xi_i) \, {\rm e}^{-\frac{v^2}{2\sigma^2}}\, v dv 
 \\
& =
\sum_{k=1}^{\infty} \dfrac{b_k}{\sigma^3 \sqrt{2\pi}}\int_{-\infty}^{\infty} v \exp\left(-\frac{v^2}{2\sigma^2}\right) \sin\left({2\pi\alpha k v}\right) dv.  
\end{align*}
Standard calculation yields 
$$
\int_{-\infty}^{\infty} v \exp\left(-\frac{v^2}{2\sigma^2}\right) \sin\left({2\pi\alpha k v}\right) dv  = \frac{2\pi \alpha k \sigma^3 \sqrt{2\pi} }{\exp(2\pi^2\alpha^2 k^2 \sigma^2)}. 
$$
Combining with \eqref{new_eq:11}, we have 
\begin{align*}
{\mathscr{D}}\left[\eta_{\sigma}(0 \, |\, \bm x, \bm \xi_i)\right]   & = \sum_{k=1}^{\infty}  \dfrac{b_k}{\sigma^3 \sqrt{2\pi}} \cdot \frac{2\pi \alpha k \sigma^3 \sqrt{2\pi} }{\exp(2\pi^2\alpha^2 k^2 \sigma^2)} =  \sum_{k=1}^{\infty} \frac{2\pi  \alpha k b_k }{\exp(2\pi^2\alpha^2 k^2 \sigma^2)}
\\
& \le \frac{\gamma_n}{2} \sum_{k=1}^{\infty} (2\pi  \alpha k )^{-n+1}\exp(- 2\pi^2\alpha^2 k^2 \sigma^2). 
\end{align*}
Observe that 
\begin{align*}
& \frac{\gamma_n}{2} \sum_{k=1}^{\infty} (2\pi  \alpha k )^{-n+1}\exp(- 2\pi^2\alpha^2 k^2 \sigma^2)
\\
 \le \, & \frac{\gamma_n}{2} (2\pi  \alpha  )^{-n+1}\exp(- 2\pi^2\alpha^2 \sigma^2) + \frac{\gamma_n}{4\pi\alpha}\int_{2\pi\alpha}^{\infty} v^{-n+1}\exp\left(-  \frac{\sigma^2 v^2}{2}\right)dv {,} 
\end{align*}
since {the function}
$v\mapsto v^{-n+1}\exp(- \sigma^2 v^2/2) $ is decreasing {in $v$} on $(0,\infty)$. 

Therefore, if $n=1$, we have 
\begin{align*}
& {\mathscr{D}}\left[\eta_{\sigma}(0 \, |\, \bm x, \bm \xi_i)\right] \le \,  \frac{\gamma_n}{2} \exp(- 2\pi^2\alpha^2 \sigma^2) + \frac{\gamma_n}{4\pi\alpha}\int_{2\pi\alpha}^{\infty} \exp\left(-  \frac{\sigma^2 v^2}{2}\right)dv
\\
&\qquad = \frac{\gamma_n}{2} \exp(- 2\pi^2\alpha^2 \sigma^2) + \frac{\gamma_n}{4\pi\alpha\sigma} \sqrt{\frac{\pi}{2}} \text{erfc}\left({\pi\alpha\sigma}{\sqrt{2}}\right)\\
& \qquad \le  \left(\frac{\gamma_n}{2} + \frac{\gamma_n}{4\alpha\sigma\sqrt{2\pi}}\right)\exp(- 2\pi^2\alpha^2 \sigma^2). 
\end{align*}
Here, erfc denotes the complementary error function. On the other hand, for $n=2$, there follows 
\begin{align*}
& {\mathscr{D}}\left[\eta_{\sigma}(0 \, |\, \bm x, \bm \xi_i)\right] \le \,  \frac{\gamma_n}{4\pi\alpha} \exp(- 2\pi^2\alpha^2 \sigma^2) + \frac{\gamma_n}{4\pi\alpha}\int_{2\pi\alpha}^{\infty} v^{-1} \exp\left(-  \frac{\sigma^2 v^2}{2}\right)dv
\\
& \qquad \le  \frac{\gamma_n}{4\pi\alpha} \exp(- 2\pi^2\alpha^2 \sigma^2)\left(1 + \frac{1}{2}\log\left(1+\frac{1}{2\pi^2\alpha^2\sigma^2}\right)\right), 
\end{align*}
where the second estimate is due to \cite[Section 5.1.20]{Handbook}. 

The case for general $n>2$ are treated by a more relaxed estimate. From integration by part, 
\begin{gather*}
\begin{aligned}
&\int_{2\pi\alpha}^{\infty} v^{-n+1}\exp\left(-  \frac{\sigma^2 v^2}{2}\right)dv 
 = \frac{1}{\sigma^2} (2\pi\alpha)^{-n} \exp(-2\pi^2 \alpha^2 \sigma^2) 
 \\
 & \qquad - \frac{n}{\sigma^2} \int_{2\pi\alpha}^{\infty} v^{-n-1}\exp\left(-  \frac{\sigma^2 v^2}{2}\right)dv \le  \frac{1}{\sigma^2} (2\pi\alpha)^{-n} \exp(-2\pi^2 \alpha^2 \sigma^2).  
\end{aligned}
\end{gather*}
We have
\begin{gather*}
\begin{aligned}
{\mathscr{D}}\left[\eta_{\sigma}(0 \, |\, \bm x, \bm \xi_i)\right] & \le  \frac{\gamma_n}{2} (2\pi  \alpha  )^{-n+1}\exp(- 2\pi^2\alpha^2 \sigma^2) +  \frac{\gamma_n}{4\pi\alpha \sigma^2} (2\pi\alpha)^{-n} \exp(-2\pi^2 \alpha^2 \sigma^2)
\\
& = \frac{\gamma_n}{2 (2\pi  \alpha  )^{n-1}} \exp(- 2\pi^2\alpha^2 \sigma^2) \left( 1 + \frac{1}{4\pi^2 \alpha^2\sigma^2}\right).
\end{aligned}
\end{gather*}

Summing the above estimates on $\{\bm \xi_1,\ldots, \bm\xi_d\}$ yields the conclusion.
\end{proof}

\begin{remark}
Proposition \ref{prop:periodic} establishes a general estimate of $\|{\nabla}_{\sigma, \bm \Xi}[\epsilon](\bm x)\| $ for $n$ times differentiable  noise. The variation of this estimate on $n$ reflects via a multiplier of $ \dfrac{\gamma_n}{2 (2\pi  \alpha  )^{n-1}}$. While the denominator increases exponentially in $n$, the bound of the n-th derivative $\gamma_n$ may also increase and counterbalance this advantage. Thus, in the upcoming analysis, we only focus on the case $n=1$ for simplicity. 
\end{remark}

Next, we consider the high-frequency bandlimited model for $\epsilon$, which is more general {than the scenario considered in the above}. Some additional definition is {given below to make the description more precise}.
Given a continuous and absolutely integrable 1D function $f$, it is known that the Fourier transform of $f$ exists on $\mathbb{R}$. We represent the Fourier transform of $f$ by $\mathcal{F}[f]$ or $\widehat{f}$, while the inverse Fourier transform is denoted by $\mathcal{F}^{-1}[f]$: 
\begin{align*}
\widehat{f}(u) = \mathcal{F}[f](u) & = \int_{-\infty}^{\infty} f(y) e^{-i 2\pi uy}dy, \ \forall u\in \mathbb{R},
\\
\mathcal{F}^{-1}[f](y) & = \int_{-\infty}^{\infty} f(u) e^{i 2\pi uy}du, \ \forall y\in \mathbb{R}.
\end{align*}
 
\begin{defi}
\label{def:bandlimited}
Let $f:\mathbb{R}\to \mathbb{R}$ be a continuous function such that $\int_{-\infty}^{\infty} |f(u)|du$ $< \infty$.  We call $f$ a high-frequency bandlimited signal if its power spectrum under Fourier transform is uniformly bounded in the frequency domain and zero below a threshold. In other words, there exist $\alpha_0$ and $\gamma>0$ such that:  
$$
\widehat{f}(u) = 0, \, \forall u\in (-\alpha_0,\alpha_0), \text{ and  }\ |\widehat{f}(u) | \le \gamma,\, \forall u\in \mathbb{R}. 
$$
We refer to $\alpha_0$ as the lowest frequency component in $f$. Also, $1/\alpha_0$ is the maximum wavelength of $f$ and a generalized notion of period for this class of signals. 
\end{defi}

The bound of ${\nabla}_{ \sigma, \bm \Xi}[\epsilon](\bm x_t)$ in this case can be established as follow. 
\begin{proposition}[high-frequency bandlimited noise]
\label{prop:high_freq}
Let $\epsilon$ be a real function on $\mathbb{R}^d$. Let $\alpha_0, \gamma >0$ and assume the cross sections $\eta(\cdot|\bm x, \bm \xi_i)$ of $\epsilon$ along $\bm \xi_1,\ldots, \bm \xi_d$ are high-frequency bandlimited signals, whose power spectrum is zero on $(-\alpha_0,\alpha_0)$ and uniformly bounded by $\gamma$ on $\mathbb{R}$, as in Definition \ref{def:bandlimited}. Then, we have 
\begin{gather}
\label{new_eq:12}
\begin{aligned}
\|{\nabla}_{ \sigma, \bm \Xi}[\epsilon](\bm x)\| & \le \frac{ \gamma\sqrt{d}}{\pi\sigma^2} \exp\left(-{2\pi^2 \alpha_0^2 \sigma^2 }\right). 
\end{aligned}
\end{gather}
\end{proposition}
\begin{proof}
We will first estimate each component of ${\nabla}_{ \sigma, \bm \Xi}[\epsilon](\bm x)$.  Recall that ${g_{\sigma}'}(v) = -\dfrac{v}{\sigma^3\sqrt{2\pi}}\exp\left(-\dfrac{v^2}{2\sigma^2}\right)$, we have 
$
\widehat{g_{\sigma}'}(u) =  i 2\pi u \exp(-{2\pi^2 \sigma^2 u^2}),\, \forall u\in \mathbb{R}.
$
By the convolution theorem, there follows 
\begin{align*}
& {\mathscr{D}}\left[\eta_{\sigma}(0 \, |\, \bm x, \bm \xi_i)\right]   = (g_{\sigma}' * \eta)(0)  = (\mathcal{F}^{-1}[\widehat{g_{\sigma}'}] * \mathcal{F}^{-1}[\widehat{\eta}]) (0) =\mathcal{F}^{-1}[\widehat{g_{\sigma}'}\, \widehat{\eta}\, ](0)
\\
&\qquad  = \mathcal{F}^{-1}\left[ i 2\pi u \exp\left(-{2\pi^2 \sigma^2 u^2}\right) \widehat{\eta}(u)\right] (0) = \int_{-\infty}^{\infty}i 2\pi u \exp\left(-{2\pi^2 \sigma^2 u^2}\right) \widehat{\eta}(u) du. 
\end{align*}
By assumption that $\eta(\cdot|\bm x,\bm\xi_i)$ are high-frequency bandlimited signals, it gives 
\begin{align*}
 {\mathscr{D}}\left[\eta_{\sigma}(0 \, |\, \bm x, \bm \xi_i)\right]  \le \, & 4\pi \gamma \int_{\alpha_0}^{\infty} u \exp\left(-{2\pi^2 \sigma^2 u^2}\right) du = \frac{ \gamma}{\pi\sigma^2} \exp\left(-{2\pi^2 \alpha_0^2 \sigma^2 }\right). 
\end{align*}
Summing the above estimate on $\{\bm \xi_1,\ldots, \bm\xi_d\}$ yields \eqref{new_eq:12}. 
\end{proof}

\subsection{Diminishing noise}
Now, we consider a diminishing noise model for $\epsilon$, in which the noise function ecays to $0$ when the optimizer reaches minimum. For that purpose, we assume $\epsilon$ is bounded by a smooth envelope function $\tilde{\epsilon}$, which is zero at $\bm x^*$: 
\begin{align}
\tilde{\epsilon}(\bm x) & = \beta \|\bm x- \bm x^*\|^2. \label{env1}
\end{align}

\begin{proposition}
\label{prop:dim_noise}
Let $\epsilon$ be a real function on $\mathbb{R}^d$ such that $|\epsilon(\bm x)| \le \tilde{\epsilon}(\bm x),\, \forall \bm x\in \mathbb{R}^d$, where $\tilde{\epsilon}(\bm x)$ is defined by \eqref{env1}. Then, we have 
\begin{gather}
\label{new_eq:13}
\begin{aligned}
\|{\nabla}_{ \sigma, \bm \Xi}[\epsilon](\bm x)\| & \le \beta \sqrt{\frac{2d}{\pi}} \left(2 \sigma + \dfrac{\|\bm x-\bm x^*\|^2}{\sigma} \right),\ \forall \bm x\in \mathbb{R}^d,
\end{aligned}
\end{gather}

\end{proposition}

\begin{proof}
For each $\bm \xi_i$ and $\bm x\in \mathbb{R}^d$, we have 
\begin{align*}
 {\mathscr{D}}\left[\eta_{\sigma}(0 \, |\, \bm x, \bm \xi_i)\right]    = \, & (g_{\sigma}' * \eta)(0)  =  \int_{-\infty}^{\infty} \eta(- v\, |\, \bm x, \bm \xi_i) \, {g_{\sigma}'}(v) dv
\\
 = & \int_{-\infty}^{\infty} \epsilon(\bm x - v \bm \xi_i ) {g_{\sigma}'}(v)dv  \le   \int_{-\infty}^{\infty}   \tilde{\epsilon}(\bm x - v \bm \xi_i ) |{g_{\sigma}'}(v)| dv .
\end{align*}
Written $\tilde{\epsilon}(\bm x - v \bm \xi_i ) = \beta \|\bm x- \bm x^* - v \bm \xi_i \|^2 = \beta \|\bm x- \bm x^*\|^2 + \beta v^2\|\bm \xi_i\|^2 - 2\beta v \langle\bm x- \bm x^* ,\bm \xi_i\rangle$, with notice that only the even part of the integrated function needs to be taken care of, there follows
\begin{gather}
\label{new_eq:14}
 \begin{aligned}   
  &  {\mathscr{D}}\left[\eta_{\sigma}(0 \, |\, \bm x, \bm \xi_i)\right]  \le\,     \frac{\beta}{\sigma^3\sqrt{2\pi}} \int_{-\infty}^{\infty}  {|v|}\exp(-\dfrac{v^2}{2\sigma^2}) (\|\bm x- \bm x^*\|^2 + v^2\|\bm \xi_i\|^2) 
  \\
  = \, &   \frac{2\beta}{\sigma^3\sqrt{2\pi}} \|\bm x- \bm x^*\|^2  \int_0^{\infty}  {v}\exp(-\dfrac{v^2}{2\sigma^2})dv 
  +  \frac{2\beta}{\sigma^3\sqrt{2\pi}}  \|\bm \xi_i\|^2 \int_0^{\infty}  {v}^3 \exp(-\dfrac{v^2}{2\sigma^2})dv
   \\
    = \, & \frac{2\beta}{\sigma^3\sqrt{2\pi}} \left(\sigma^2  \|\bm x- \bm x^*\|^2 + 2\sigma^4 \|\bm \xi_i\|^2 \right) = \beta \sqrt{\frac{2}{\pi}} \left( \frac{ \|\bm x- \bm x^*\|^2}{\sigma} + 2\sigma 
    \right) .
 \end{aligned}
 \end{gather}
Summing \eqref{new_eq:14} on $\{\bm \xi_1,\ldots, \bm\xi_d\}$ gives \eqref{new_eq:13}.
\end{proof}
\begin{remark}
The estimate \eqref{new_eq:13} is sharp in the sense that for a fixed $\bm x_0$, we can find a noise function $\epsilon:\mathbb{R}^d\to \mathbb{R}$ satisfying $|\epsilon(\bm x)| \le \tilde{\epsilon}(\bm x),\, \forall \bm x\in \mathbb{R}^d$ such that equality in \eqref{new_eq:13} occurs at $\bm x_0$. While the envelope bound \eqref{env1} seems to be loose, especially away from the global minimum, it is the decaying rate of the envelope in the near neighborhood of $\bm x^*$ that determines the dependence of $\|{\nabla}_{ \sigma, \bm \Xi}[\epsilon](\bm x)\| $ on $\sigma$ near $\bm x^*$. 

More complicated examples of the envelope of the noise can be considered with more involved calculation. Below, we present such an example in case $d=1$
\begin{align}
\tilde{\epsilon}(x) & = 1 - \exp(-\beta(x-x^*)^2),\, \forall x \in \mathbb{R} , \label{env2}
\end{align}
($x$ and $x^*$ are now scalar, hence regular font). This bound enjoys similar decaying rate towards $x^*$ as in \eqref{env1}, but is significantly more restrictive globally. We show that 
\begin{gather}
\label{new_eq:16}
\begin{aligned}
|{\nabla}_{\sigma}[\epsilon](x)| & \le  \sqrt{\frac{2}{\pi}} \left(\sigma + \dfrac{\beta ( x- x^*)^2}{\sigma} \right),\ \forall  x\in \mathbb{R}. 
\end{aligned}
\end{gather}
which is similar to Proposition \ref{prop:dim_noise}, except for a weaker dependence on $\beta$.


To prove \eqref{new_eq:16}, note that 
\begin{align*}
 & |{\nabla}_{\sigma}[\epsilon](x)|  =   \int_{-\infty}^{\infty} \epsilon( x - v ) {g_{\sigma}'}(v)dv \le    \int_{-\infty}^{0} \tilde{\epsilon}( x - v ) {g_{\sigma}'}(v)dv  -  \int_{0}^{\infty} \tilde{\epsilon}( x - v ) {g_{\sigma}'}(v)dv 
 \end{align*}
 The first term can be estimated as 
 \begin{align*}
 &  \int_{-\infty}^{0} \tilde{\epsilon}( x - v ) {g_{\sigma}'}(v)dv  =  - \dfrac{1}{{\sigma^3\sqrt{2\pi}}} \int_{-\infty}^{0} [1 - \exp(-\beta(x -x^* - v)^2)] \exp\left(-\dfrac{v^2}{2\sigma^2}\right) v dv
 \\
    & = \left((x-x^*) \sqrt{\pi \beta {\sigma^2(\sigma^2+1)}}  \exp\left(\frac{\beta \sigma^2 (x-x^*)^2}{\sigma^2+1}\right) \emph{erfc}\left(\frac{\sigma (x-x^*)\sqrt{\beta}}{\sqrt{\sigma^2+1}}\right) - \sigma^2 - 1 \right)   
    \\
    & \qquad\qquad \cdot  \frac{1}{\sigma^3\sqrt{2\pi}}\cdot\frac{\sigma^2}{(\sigma^2+1)^2}\cdot \exp(-\beta(x-x^*)^2) + \frac{1}{\sigma^3\sqrt{2\pi}}\cdot \sigma^2
    \\
    & = \frac{(x-x^*)\sqrt{\beta}  }{\sqrt{2}(\sigma^2+1)^{3/2}}  \exp\left(- \frac{ \beta (x-x^*)^2}{\sigma^2+1}\right) \emph{erfc}\left(\frac{\sigma (x-x^*)\sqrt{\beta}}{\sqrt{\sigma^2+1}}\right) 
    \\
    & \qquad \qquad  -  \frac{1}{\sigma(\sigma^2+1)\sqrt{2\pi}}\cdot  \exp(-\beta(x-x^*)^2)  + \frac{1}{\sigma\sqrt{2\pi}}. 
\end{align*}
Similarly, 
\begin{align*}
 &  -  \int_{0}^{\infty} \tilde{\epsilon}( x - v ) {g_{\sigma}'}(v)dv   =  \frac{1}{\sigma\sqrt{2\pi}} - \frac{1}{\sigma(\sigma^2+ 1)\sqrt{2\pi}} \cdot \exp(-\beta(x-x^*)^2)  
     \\
     &  \qquad \qquad - \frac{(x-x^*)\sqrt{\beta} }{\sqrt{2}(\sigma^2+1)^{3/2}}    \exp\left(-\frac{\beta(x-x^*)^2}{\sigma^2 +1}\right) \left(2 - \emph{erfc}\left({\frac{\sigma(x-x^*)\sqrt{\beta}}{\sqrt{\sigma^2+1}}}\right)  \right).
\end{align*}
There follows
\begin{align*}
       & |{\nabla}_{\sigma}[\epsilon](x)|   \le \,  \frac{\sqrt{2}}{\sigma\sqrt{\pi}} - \frac{\sqrt{2}}{\sigma(\sigma^2+ 1)\sqrt{\pi}} \cdot \exp(-\beta(x-x^*)^2)
        \\
      &  \qquad\qquad  -   \frac{(x-x^*)\sqrt{2\beta}  }{(\sigma^2+1)^{3/2}}  \exp\left(- \frac{ \beta(x-x^*)^2}{\sigma^2+1}\right) \emph{erf}\left(\frac{\sigma (x-x^*)\sqrt{\beta}}{\sqrt{\sigma^2+1}}\right).
\end{align*}
Since $-\exp(-\beta(x-x^*)^2)  < - 1 + \beta {(x-x^*)^2} $, we have
\begin{align*}
     |{\nabla}_{\sigma}[\epsilon](x)|   \le \frac{\sqrt{2}}{\sigma\sqrt{\pi}} - \frac{\sqrt{2}}{\sigma(\sigma^2+ 1)\sqrt{\pi}} (1 - \beta{(x-x^*)^2} ) \le  \sqrt{\frac{2}{\pi}} \left(\sigma + \frac{\beta (x-x^*)^2}{\sigma} \right),
\end{align*}
as desired.
\end{remark}

\section{Convergence analysis}
\label{sec:analysis}

This section {contains} the main results of our paper. We establish the convergence rate of Algorithm \eqref{alg:GD} in two scenarios described in the last section, and show the advantage of DGS gradient when using moderate or large smoothing radius $\sigma$. The interested reader is referred to \cite{DBLP:conf/uai/ZhangTLZ21} for the asymptotic consistency of the DGS gradient (i.e., its convergence to local gradient when $\sigma$ approaches to $0$) and the error analysis of the method in case of small smoothing radius. 

Let us define the discrepancy between the DGS gradient and the gradient of $\phi$ at iteration $t$: 
\begin{align}
\label{def:B}
    B_t =  \| \widetilde{\nabla}^M_{ \sigma, \bm \Xi}[F](\bm x_t) -  {\nabla}\phi(\bm x_t)\|^2.  
\end{align}
$\nabla \phi$ represents the search direction that the optimizer should follow. Since $\phi$ is not accessible, we cannot compute $\nabla\phi$ directly, but rely on the DGS gradient, which uses Gaussian smoothing to filter out the noise in the observation, as a reliable surrogate for $\nabla \phi$. Intuitively, the better $\widetilde{\nabla}^M_{ \sigma, \bm \Xi}[F](\bm x_t)$ can approximate ${\nabla}\phi(\bm x_t)$, the faster convergence we can expect of the gradient scheme  \eqref{alg:GD}.  

Let $r_t$ denote the distance between the iterate $\bm x_t$ and the global minimum $\bm x^*$: 
\begin{align*}
r_t = \|\bm x_t - \bm x^*\|, 
\end{align*}
we first have the following general {result} establishing an upper bound on $r_t$ at each iteration in terms of $B_t$.
\begin{lemma}
\label{lemma:bound_rt}
Consider {Algorithm} \eqref{alg:GD}. Assume $\lambda \le \dfrac{1}{8L}$. Then, at iteration $t$, we have 
\begin{align}
\label{new_eq:0}
  r_{t+1}^2  {\le} \,      
 (1 - (\lambda \tau -8\lambda^2 \tau L)) r_{t}^2
 +  \left(\frac{2\lambda}{\tau}  +  2 \lambda^2 \right) B_t. 
\end{align}
\end{lemma}

\begin{proof}
First, we have for all $\bm x \in \mathbb{R}^d$
\begin{gather}
\begin{aligned}
\|\widetilde{\nabla}^M_{\sigma, \bm \Xi}[F](\bm x)]\|^2  & =  \|\widetilde{\nabla}^M_{\sigma, \bm \Xi}[F](\bm x)] - \nabla\phi(\bm x) + \nabla\phi(\bm x) \|^2
    \\
     & \le\, 2 \|\widetilde{\nabla}^M_{\sigma, \bm \Xi}[F](\bm x)] - \nabla \phi(\bm x)\|^2 + 2 \|\nabla \phi(\bm x)\|^2.
\end{aligned}
\label{new_eq:1}
\end{gather}

By the definition of $r_t$, {we get}
\begin{align}
&\qquad r_{t+1}^2   = \, \|\bm x_t - \lambda \widetilde{\nabla}^M_{ \sigma, \bm \Xi}[F](\bm x_t) - \bm x^*\|^2  \notag 
 \\
& = \, r_{t}^2 - 2\lambda \langle \widetilde{\nabla}^M_{ \sigma, \bm \Xi}[F](\bm x_t), \bm x_t - \bm x^* \rangle + \lambda^2 \| \widetilde{\nabla}^M_{ \sigma, \bm \Xi}[F](\bm x_t)\|^2 \notag
\\
&\! \overset{\eqref{new_eq:1}}{\le} \, r_{t}^2 - 2\lambda \langle \widetilde{\nabla}^M_{ \sigma, \bm \Xi}[F](\bm x_t)  -  {\nabla}\phi(\bm x_t), \bm x_t - \bm x^* \rangle 
 - 2\lambda \langle {\nabla}\phi(\bm x_t), \bm x_t - \bm x^* \rangle \notag
 \\
&  \qquad \qquad + 2 \lambda^2 \|\widetilde{\nabla}^M_{\sigma, \bm \Xi}[F](\bm x_t)] - \nabla \phi(\bm x_t)\|^2 + 2 \lambda^2  \|\nabla \phi(\bm x_t)\|^2{.}  \label{new_eq:2}
\end{align}

We {then} proceed to bound the right hand side of \eqref{new_eq:2}. {Since} $\phi$ is strongly convex, 
\begin{gather}
\label{new_eq:3}
\begin{aligned}
      & - 2\lambda \langle\nabla \phi(\bm x_t) , \bm x_t - \bm  x^* \rangle \le 2\lambda \phi(\bm x^*) 
 - 2\lambda \phi(\bm x_t)  - {\lambda \tau}\|\bm x^* - \bm x_t\|^2. 
\end{aligned}
\end{gather}

On the other hand, 
\begin{align}
&   - 2\lambda \langle \widetilde{\nabla}^M_{ \sigma, \bm \Xi}[F](\bm x_t) -  {\nabla}\phi(\bm x_t), \bm x_t - \bm x^* \rangle     \label{new_eq:4}
\\
& \qquad \le\,  \frac{2\lambda}{\tau} \| \widetilde{\nabla}^M_{ \sigma, \bm \Xi}[F](\bm x_t) -  {\nabla}\phi(\bm x_t)\|^2 + \frac{\lambda\tau}{2} \| \bm x_t - \bm x^* \|^2{,} \notag
\end{align}

Applying an estimate for convex, $\mathbb{C}^{1,1}$-functions {(see, e.g., Theorem 2.1.5, \cite{Nesterov_book_2004}), we have}
\begin{align}
  & 2\lambda^2 \| \nabla\phi (\bm x)\|^2 
 \le 16 \lambda^2 L(\phi(\bm x_t) - \phi(\bm x^*)). \label{new_eq:5}
\end{align}

Since $\phi$ is a strongly convex function, we have that 
\begin{align}
    & - (\phi(\bm x_t) -   \phi(\bm x^*)) 
       \le \, - \, \frac{\tau}{2}\|\bm x_t - \bm x^*\|^2. \label{new_eq:6}
\end{align}

Combining \eqref{new_eq:2}--\eqref{new_eq:6}, there holds 
\begin{gather}
  \label{new_eq:7}
\begin{aligned}
  r_{t+1}^2   \,& {\le} \, r_{t}^2     
 - (2\lambda -16\lambda^2 L)( \phi(\bm x_t) -   \phi(\bm x^*) )  
  +  \left(\frac{2\lambda}{\tau}  +  2 \lambda^2 \right) B_t
 \\
 &{\le} \,      
 (1 - (\lambda \tau -8\lambda^2 \tau L)) r_{t}^2
 +  \left(\frac{2\lambda}{\tau}  +  2 \lambda^2 \right) B_t. 
\end{aligned}
\end{gather}

We arrive at the conclusion.
\end{proof}

{The lemma below provides an important estimate of $B_t$ for the next analysis.}


\begin{lemma}
\label{Lemma:bound_Bt}
Let $F$ be a function defined on $\mathbb{R}^d$ {that} can be represented as $F=\phi + \epsilon$, where $\phi$ satisfies Assumption \ref{assump:lipschitz_cont}. Then we have 
\begin{align}
    \label{new_eq:8}
    B_t \le\, C \frac{ {\pi} (M\,!)^2 d}{4^M ((2M)\,!)^2} \sigma^{4M-2} + 48  L^2 d \sigma^2 + 3 \|{\nabla}_{ \sigma, \bm \Xi}[\epsilon](\bm x_t)\|^2,\, \forall t > 0. 
\end{align}
\end{lemma}
\begin{proof}
For {a} unit vector $\bm \xi \in \mathbb{R}^d$, define $\nabla_{\bm \xi} \phi(\bm x)$ the partial derivatives of $\phi$ at $\bm x\in \mathbb{R}^d$ in direction $\bm \xi$. {With a slight abuse of notation, we denote $\phi_\sigma(y|\bm x,\bm \xi)$} the 1D Gaussian smoothing of $\phi(\bm x)$ along the direction $\bm \xi$ with {a} smoothing radius $\sigma$, i.e., $\phi_\sigma(y|\bm x,\bm \xi) = \mathbb{E}_{v \sim \mathcal{N}(0, 1)} \left[\phi(\bm x + (y + \sigma v) \bm \xi) \right]$. By definition \eqref{def:B} of $B_t$, 
\begin{gather}
\label{new_eq:8b}
\begin{aligned}
         B_t &  = \| \widetilde{\nabla}^M_{ \sigma, \bm \Xi}[F](\bm x_t) -  {\nabla}\phi(\bm x_t)\|^2 \le 3 \| \widetilde{\nabla}^M_{ \sigma, \bm \Xi}[F](\bm x_t) -  {\nabla}_{ \sigma, \bm \Xi}[F](\bm x_t)\|^2
        \\
        &\qquad   + 3 \| {\nabla}_{ \sigma, \bm \Xi}[\phi](\bm x_t) -  {\nabla}\phi(\bm x_t)\|^2 + 3 \|{\nabla}_{ \sigma, \bm \Xi}[\epsilon](\bm x_t)\|^2
        \\
        & = 3\sum_{i=1}^d \left| \widetilde{\mathscr{D}}^M\left[G_{\sigma}(0 \, |\, \bm x_t, \bm \xi_i)\right] - {\mathscr{D}}\left[G_{\sigma}(0 \, |\, \bm x_t, \bm \xi_i)\right]\right|^2 
        \\
        & \qquad + 3\sum_{i=1}^d \left|{\mathscr{D}}\left[\phi_{\sigma}(0 \, |\, \bm x_t, \bm \xi_i)\right] - \nabla_{\bm \xi} \phi(\bm x) \right|^2 + 3 \|{\nabla}_{ \sigma, \bm \Xi}[\epsilon](\bm x_t)\|^2. 
\end{aligned}
\end{gather}
We proceed to bound the first and second terms in the above right hand side. First, the approximation error of the Gauss-Hermite formula can be bounded as 
\begin{align}
\label{GH_error}
\left|\widetilde{\mathscr{D}}^M[G_\sigma(0 \, |\, \bm x_t, \bm \xi_i)] - \mathscr{D}[G_\sigma(0 \, |\, \bm x_t, \bm \xi_i)] \right| \le C \frac{ M\,!\sqrt{\pi}}{2^M(2M)\,!} \sigma^{2M-1}, 
\end{align}
where the constant $C>0$ is independent of $M$ and $\sigma$ \cite{Handbook}. On the other hand, an adaptation of \cite[Lemma 3]{NesterovSpokoiny15} to {the} 1D Gaussian smoothing gives 
\begin{align}
    \label{new_eq:9}
    \ \left|{ \mathscr{D}}\left[\phi_{\sigma}(0 \, |\, \bm x_t, \bm \xi_i)\right] - \nabla_{\bm \xi_i} \phi(\bm x_t) \right| \le 4\sigma L. 
\end{align}
Summing \eqref{GH_error} and \eqref{new_eq:9} from $i=1$ to $d$ and combining with \eqref{new_eq:8b} conclude the lemma. 
\end{proof}

As we see in \eqref{new_eq:8}, the discrepancy between the DGS search direction and $\nabla \phi$ can be broken down into three terms, correspondingly originating from: i) the use of GH quadrature to approximate the DGS gradient; ii) the difference between DGS gradient and standard gradient of $\phi$; and iii) the perturbation coming from the DGS gradient of the noise function $\epsilon$ (which has been estimated in Section \ref{sec:statement}). 


\subsection{Wave-like noise}$ $
{We first consider the periodic noise model. Assume that the noise function satisfies the conditions given}
in Proposition \ref{prop:periodic}, 
{then} $B_t$ is uniformly bounded for all $t\ge 0$. The following result can be implied from Lemma \ref{lemma:bound_rt}, by applying \eqref{new_eq:0} recursively.

\begin{lemma}
\label{cor:lemma:bound_rt1}
Consider {Algorithm} \eqref{alg:GD} with a fixed step size $ \lambda < \dfrac{1}{8L}$. Assume $B>0$ satisfies $B_t \le B, \forall t> 0$. Then, at iteration $t$, we have 
\begin{align*}
  r_{t+1}^2  {\le} \,      
  (1 - (\lambda \tau -8\lambda^2 \tau L)) ^{t+1} r_0^2 + \frac{ \left({2}  +  2 \lambda \tau \right) B}{( \tau^2 -8\lambda \tau^2 L)}. 
\end{align*}
\end{lemma}

We achieve the following convergence rate result for the gradient descent scheme \eqref{alg:GD} using our DGS gradient. 
\begin{theorem}[Convergence on periodic noise models]
\label{thm:main}
Let $F = \phi + \epsilon$, where $\phi$ satisfies Assumptions \ref{assump:lipschitz_cont}--\ref{assump:strongly_convex} and $\epsilon$ is a real function on $\mathbb{R}^d$ whose cross sections $\eta(\cdot|\bm x, \bm \xi_i)$ of $\epsilon$ along $\bm \xi_1,\ldots, \bm \xi_d$ are periodic functions with {a} period $1/\alpha$. Assume that $\eta(\cdot|\bm x,\bm\xi_i)$ are continuously differentiable with uniformly bounded derivatives $|\eta' (y|\bm x,\bm \xi_i)|\, \le \gamma_1,\ \forall y\in \mathbb{R},\ i\in \{1,\ldots,d\}$. Let $\{\bm x_t\}_{t\ge 0}$ be generated by Algorithm \eqref{alg:GD} with $\lambda = {1}/{(16L)}$ and $\bm x^*$ be the global minimum of $\phi$. Then, for any $t\ge 0$, we have 
\begin{align}
        & \|\bm x_{t+1} - \bm x^*\|^2  \le\,  
 \left(1 - \frac{\tau}{32L}\right) ^{t+1} \|\bm x_0 - \bm x^*\|^2 +  \delta_{\sigma}     , \label{thm:strong_convex}
\end{align}
where
\begin{align}
\label{new_eq:15c}
&  \delta_{\sigma} =  
\left(\!\frac{4}{ \tau^2}  +  \frac{1}{4 L\tau} \! \right) \!\! \left( \frac{ C {\pi} (M\,!)^2 d}{4^M ((2M)\,!)^2} \sigma^{4M-2} + 48  L^2 d \sigma^2 + \frac{3 \gamma_1^2 {d}   }{2 e^{ 4\pi^2\alpha^2 \sigma^2}}   \left( {1} + \frac{1}{8\pi \alpha^2\sigma^2}\right)\right)\!.
\end{align} 
Further, $\sigma$ can be selected to suit the noise frequency and minimize $\delta_\sigma$ as follows: 
\begin{itemize}
    \item If $\alpha \! >\! \dfrac{2L\sqrt{2}}{\pi\gamma_1}\!$ (small period, 
{high}
frequency): choose $\sigma  =  \dfrac{1}{\pi\alpha\sqrt{2}}\log^{1/2}\!\!\left(\dfrac{\pi\alpha \gamma_1}{2L\sqrt{2}}\right)$ then
\begin{gather}
\label{new_eq:15}
\begin{aligned}
\delta_{\sigma} &\le  \left(\!\frac{4}{ \tau^2}  +  \frac{1}{4 L\tau}\! \right) \Bigg[ \frac{ C {\pi} (M\,!)^2 d}{4^M ((2M)\,!)^2} \sigma^{4M-2} 
\\
& \qquad \qquad + \frac{30 L^2 d}{\pi^2 \alpha^2} \left(\log\left(\dfrac{\pi\alpha \gamma_1}{2L\sqrt{2}}\right) + \log^{-1}\left(\dfrac{\pi\alpha \gamma_1}{2L\sqrt{2}}\right)  \right) \Bigg].
\end{aligned}
\end{gather}
    \item If $\alpha \le  \dfrac{2L\sqrt{2}}{\pi\gamma_1}$ (large period, 
{low}
frequency): choose $\sigma = \dfrac{1}{\alpha}$, then
\begin{gather}
\label{new_eq:15b}
\begin{aligned}
\delta_{\sigma} \le  \left(\frac{4}{ \tau^2}  +  \frac{1}{4 L\tau}\! \right) \left[ \frac{ C {\pi} (M\,!)^2 d}{4^M ((2M)\,!)^2} \sigma^{4M-2} + \frac{64 L^2 d}{\alpha^2} 
\right].
\end{aligned}
\end{gather}
\end{itemize}
\end{theorem}

\begin{proof}
Define $B_t$ as in \eqref{def:B}. From Lemma \ref{Lemma:bound_Bt} and Proposition \ref{prop:periodic}, $B_t$ can be uniformly bounded as 
\begin{gather}
\begin{aligned}
 B_t  \le\, &  \frac{ C{\pi} (M\,!)^2 d}{4^M ((2M)\,!)^2} \sigma^{4M-2} + 48  L^2 d \sigma^2 + 3 \|{\nabla}^M_{ \sigma, \bm \Xi}[\epsilon](\bm x_t)\|^2 \label{new_eq:17}
\\
\le\, & \frac{ C {\pi} (M\,!)^2 d}{4^M ((2M)\,!)^2} \sigma^{4M-2} + 48  L^2 d \sigma^2 +  \frac{3 \gamma_1^2 {d}   }{2 e^{ 4\pi^2\alpha^2 \sigma^2}}   \left( {1} + \frac{1}{8\pi \alpha^2\sigma^2}\right),\ \forall t > 0.  
\end{aligned}
\end{gather}
Let $B$ be the upper bound of $B_t$ in \eqref{new_eq:17}, from Lemma \ref{cor:lemma:bound_rt1}, we have 
\begin{align*}
\|\bm x_{t+1} - \bm x^*\|^2  & \le  (1 - (\lambda \tau -8\lambda^2 \tau L)) ^{t+1} \|\bm x_0 - \bm x^*\|^2 + \frac{ \left({2}  +  2 \lambda \tau \right) B}{( \tau^2 -8\lambda \tau^2 L)}
\\
& \le \left(1 - \frac{\tau}{32L}\right) ^{t+1} \|\bm x_0 - \bm x^*\|^2 +  \left(\frac{4}{ \tau^2}  +  \frac{1}{4 L\tau} \right) B,   
\end{align*}
proving \eqref{thm:strong_convex}. 

Now, we consider two scenarios of small and large period of the noise. If $\alpha > \dfrac{2L\sqrt{2}}{\pi\gamma_1}$, from \eqref{thm:strong_convex}, the optimal $\sigma$ is approximately the minimum of function 
$
\sigma \mapsto 48  L^2 d \sigma^2 + \dfrac{3 \gamma_1^2 {d}   }{2 e^{ 4\pi^2\alpha^2 \sigma^2}}    , 
$
which is 
$
\sigma  =  \dfrac{1}{\pi\alpha\sqrt{2}}\log^{1/2}\left(\dfrac{\pi\alpha \gamma_1}{2L\sqrt{2}}\right). 
$ 
With this $\sigma$, we have
\begin{align*}
48  L^2 d \sigma^2 \le \frac{24  L^2 d }{\pi^2\alpha^2} \log\left(\dfrac{\pi\alpha \gamma_1}{2L\sqrt{2}}\right), & 
\quad \frac{3 \gamma_1^2 {d}   }{2 e^{ 4\pi^2\alpha^2 \sigma^2}} \le \frac{3\gamma_1^2 d}{2} \cdot  \frac{8L^2}{\pi^2 \alpha^2 \gamma_1^2} =   \frac{12  L^2 d }{\pi^2\alpha^2} ,
\\
{1} + \frac{1}{8\pi \alpha^2\sigma^2} &\le 1 + \frac{\pi}{4}  \log^{-1}\left(\dfrac{\pi\alpha \gamma_1}{2L\sqrt{2}}\right),
\end{align*}
and \eqref{new_eq:15} follows directly from the definition of $\delta_\sigma$.

If $\alpha \le  \dfrac{2L\sqrt{2}}{\pi\gamma_1}$, with $\sigma = \dfrac{1}{\alpha}$,    
\begin{align*}
&48  L^2 d \sigma^2 \le \frac{48L^2 d}{\alpha^2}, \ \  \frac{3 \gamma_1^2 {d}   }{2 e^{ 4\pi^2\alpha^2 \sigma^2}} \le\frac{24L^2 d}{\pi^2 \alpha^2}, \ \ {1} + \frac{1}{8\pi \alpha^2\sigma^2} < 2, 
\end{align*}
implying \eqref{new_eq:15b}
\end{proof}

The estimate \eqref{thm:strong_convex} showed that the gradient descent scheme \eqref{alg:GD} is linearly convergent, whose rate only depends on the Lipschitz constant and the convexity parameter of $\phi$. However, as typically seen in optimization under the presence of noise, Algorithm \eqref{alg:GD} can only converge to a neighborhood of $\bm x^*$ and not exactly $\bm x^*$. The radius of this neighborhood of convergence is $\sqrt{\delta_\sigma}$ and determined by the discrepancy between $\nabla \phi$ and the search direction formed by DGS gradients (see \eqref{new_eq:15c}). There, the first term stems from the error of GH estimator and can be easily controlled by assigning enough GH quadrature nodes (usually a small number) without adjusting the smoothing radius. The latter two terms, on the other hand, reveal the need of choosing the smoothing radius appropriately for Algorithm \eqref{alg:GD}: 
with too large $\sigma$, the DGS gradient may not adequately approximate the true slope of $\phi$ at the point of interest (resulting in large second term), while with too small $\sigma$, it cannot efficiently filter out the noise (reflected in large third term). The optimal range of the smoothing radius was determined to balance these two opposite behaviors and minimize $\delta_\sigma$ (see \eqref{new_eq:15}-\eqref{new_eq:15b}). Interestingly, in both large and small period scenarios, we proved that the optimal choices of $\sigma$ are roughly $1/\alpha$, the period of the noise function. Further, with these choices, the neighborhood of convergence can be tightened to have the radius of approximately $\sqrt{d}/\alpha$, up to a log factor (equivalently, $1/\alpha$ in each direction). After reaching this neighborhood, the optimizer 
{is no longer assured to} make further progress because the actual global minimum may be concealed by spurious local minima of the noise function on the interval of one period. Without additional assumption on the behavior of the noise on this interval (e.g., diminishing around the global minimum), this barrier may not be relaxed in theory. Note that the smaller period and higher frequency the noise function has, the better we can approximate to the global minima of objective function.

The error estimate \eqref{thm:strong_convex} shares similarities with several existing results of optimization of noisy functions \cite{doi:10.1137/110830629,doi:10.1137/18M1177718,doi:10.1137/19M1240794,2022arXiv220100973S}, which featured the convergence to a neighborhood around the global minimum with size depending on the noise. DGS can be viewed as a generalization of finite difference, where we use more function evaluations on each direction to estimate the directional derivatives \textit{with noise filtered out}. From this angle, the smoothing radius $\sigma$ can be seen as the counterpart of difference interval $h$ in finite difference schemes. In \cite{Mor2012EstimatingDO,BCCS19}, it has been showed that the optimal neighborhood of convergence requires a precise specification of $h$, with too big or too small $h$ leading to either inadequate gradient approximation of the true objective function or insufficient noise removal, which is in line with \eqref{new_eq:15c}. However, our analysis has a key novelty: here, \textit{both the optimal smoothing radius and the size of the neighborhood of convergence are determined based mainly on the period and frequency of the noise, and depend very weakly on its magnitude} (via the log factor or totally independent). By contrast, in all the aforementioned works, the choice of optimal parameters and the convergence study were centered around on the noise magnitude upper bound (in the uniform or expectation sense), which is in general more pessimistic. It may be natural and advantageous to exploit the frequency information in optimization with oscillating noise. To the best of our knowledge, DGS approach is the first to offer that ability with a rigorous theory.

Next, we extend Theorem \ref{thm:main} to a more general setting, where the noise function is assumed to be high-frequency bandlimited instead of periodic. We will see that the gradient descent scheme \eqref{alg:GD} achieves similar exponential convergence rate, and the optimal smoothing radius is roughly $1/\alpha_0$, which is the maximum
wavelength of the noise and a generalization of the notion of period for this model. Unsurprisingly, the size of the neighborhood of convergence is also approximately $\sqrt{d}/\alpha_0$ with this value of $\sigma$. Thus, the smaller the maximum wavelength is, the closer Algorithm \eqref{alg:GD} can reach to the global minimum.   
%
$ $
\begin{theorem}[Convergence on high-frequency bandlimited noise models]
\label{thm:main2}
Let $F = \phi + \epsilon$, where $\phi$ satisfies Assumptions \ref{assump:lipschitz_cont}--\ref{assump:strongly_convex} and $\epsilon$ is a real function on $\mathbb{R}^d$ whose cross sections $\eta(\cdot|\bm x, \bm \xi_i)$ of $\epsilon$ along $\bm \xi_1,\ldots, \bm \xi_d$ are high-frequency bandlimited signals with power spectrum being zero on $(-\alpha_0,\alpha_0)$ and uniformly bounded by $\gamma$ on $\mathbb{R}$ (see Definition \ref{def:bandlimited}). Let $\{\bm x_t\}_{t\ge 0}$ be generated by \eqref{alg:GD} with $\lambda = {1}/{(16L)}$ and $\bm x^*$ be the global minimum of $\phi$. Then, for any $t\ge 0$, we have 
\begin{align}
        & \|\bm x_t - \bm x^*\|^2  \le\,  
 \left(1 - \frac{\tau}{32L}\right) ^{t+1} \|\bm x_0 - \bm x^*\|^2 +  \delta_{\sigma}     , \label{thm:strong_convex2}
\end{align}
\begin{align*}
 \text{with}\ \   & \delta_{\sigma} =  
\left(\frac{4}{ \tau^2}  +  \frac{1}{4 L\tau} \right) \! \left( \frac{ C{\pi} (M\,!)^2 d}{4^M ((2M)\,!)^2} \sigma^{4M-2} + 48  L^2 d \sigma^2 + \frac{ 3 \gamma^2{d}}{\pi^2\sigma^4} \exp\left(-{4\pi^2 \alpha_0^2 \sigma^2 }\right)\right).
\end{align*} 
Further, $\sigma$ can be selected to suit the noise frequency and minimize $\delta_\sigma$ as follows: 
\begin{itemize}
    \item If $\alpha_0 > L^{1/3}/(\pi \gamma^{1/3})$: choose $\sigma = \dfrac{\sqrt{3}}{\pi\alpha_0\sqrt{2}}\log^{1/2}\left(\dfrac{\pi\alpha_0 \gamma^{1/3}}{L^{1/3}}\right)$ then
\begin{gather}
\label{new_eq:24}
\begin{aligned}
\delta_{\sigma} &\le  \left(\!\frac{4}{ \tau^2}  +  \frac{1}{4 L\tau}\! \right) \Bigg[ \frac{ C {\pi} (M\,!)^2 d}{4^M ((2M)\,!)^2} \sigma^{4M-2}  
\\
& \qquad + \frac{8 L^2 d }{  \alpha_0^2} \left(\log\left(\dfrac{\pi\alpha_0 \gamma^{1/3}}{L^{1/3}}\right) +  \log^{-2}\left(\dfrac{\pi\alpha_0 \gamma^{1/3}}{L^{1/3}}\right) \right)\Bigg].
\end{aligned}
\end{gather}
    \item If $\alpha_0 \le L^{1/3}/(\pi \gamma^{1/3})$: choose $\sigma = \dfrac{1}{\alpha_0}$, then
\begin{gather}
\label{new_eq:24b}
\begin{aligned}
\delta_{\sigma} \le  \left(\frac{4}{ \tau^2}  +  \frac{1}{4 L\tau}\! \right) \left[ \frac{ C {\pi} (M\,!)^2 d}{4^M ((2M)\,!)^2} \sigma^{4M-2} +  \frac{49 L^2 d}{\alpha_0^2}
\right].
\end{aligned}
\end{gather}
\end{itemize}
\end{theorem}
\begin{proof}
Estimate \eqref{thm:strong_convex2}  
{follows from}
Proposition \ref{prop:high_freq}, Lemmas \ref{Lemma:bound_Bt} and \ref{cor:lemma:bound_rt1}. To determine the optimal value of $\sigma$ such that Algorithm {\eqref{alg:GD}} approaches the true minimum as close as possible, we find the critical point of the function $\sigma \mapsto  48  L^2 d \sigma^2 + \frac{ 3 \gamma^2{d}}{\pi^2\sigma^4} \exp(-{4\pi^2 \alpha_0^2 \sigma^2 }) $, which satisfies 
\begin{align*}
48L^2 d -\frac{ 3 \gamma^2{d}}{\pi^2}  \left(\frac{2}{\sigma^6} + \frac{4\pi^2 \alpha_0^2}{\sigma^4}\right) e^{- 4\pi^2\alpha_0^2 \sigma^2} = 0, 
\end{align*}
thus can be approximated as 
\begin{align*}
48L^2 d -\frac{3 \gamma^2 {d}  \left({2} + {4\pi^2 \alpha_0^2}\right)  }{\pi^2\sigma^4 e^{4\pi^2\alpha_0^2 \sigma^2} }= 0, \ \text{ or }\ \sigma^4 e^{4\pi^2\alpha_0^2 \sigma^2} = \frac{ \gamma^2   \left({2} + {4\pi^2 \alpha_0^2}\right) }{16 \pi^2 L^2  }.
\end{align*}
Simple calculation gives    
$$
\sigma = \frac{1}{\pi\alpha_0 \sqrt{2}} \left[W_0\left(2\pi^2\alpha_0^2 \sqrt{\frac{\gamma^2 (2\pi^2\alpha_0^2 + 1)}{8 \pi^2 L^2}}\right)\right]^{1/2}, 
$$
where $W_0$ is the Lambert $W$ function. Observe {that} $W_0$ is an increasing function and $W_0(u) \simeq \log(u) - \log(\log u),\, \forall u \ge e $, \cite{Hoorfar_inequalitieson}, we consider two scenarios similarly to the periodic noise case: 
\begin{enumerate}
\item $\alpha_0 > L^{1/3}/(\pi \gamma^{1/3})$: we choose $\sigma = \dfrac{\sqrt{3}}{\pi\alpha_0\sqrt{2}}\log^{1/2}\left(\dfrac{\pi\alpha_0 \gamma^{1/3}}{L^{1/3}}\right)$. Then
\end{enumerate}
\begin{gather}
\label{new_eq:25}
\begin{aligned}
& 48  L^2 d \sigma^2  \le \frac{72 L^2 d }{ \pi^2 \alpha_0^2} \log\left(\dfrac{\pi\alpha_0 \gamma^{1/3}}{L^{1/3}}\right),  
\\  
 & \frac{ 3 \gamma^2{d}}{\pi^2\sigma^4} \exp\left(-{4\pi^2 \alpha_0^2 \sigma^2 }\right)  
 \le  \frac{ 3 \gamma^2{d}}{\pi^2} \cdot \frac{4 \pi^4 \alpha_0^4}{9} \log^{-2}\left(\dfrac{\pi\alpha_0 \gamma^{1/3}}{L^{1/3}}\right) \cdot \frac{L^2}{\gamma^2 \pi^6 \alpha_0^6}
 \\
 &\qquad\qquad\qquad\qquad\quad\ \ = \frac{4L^2 d}{3\pi^4 \alpha_0^2}  \log^{-2}\left(\dfrac{\pi\alpha_0 \gamma^{1/3}}{L^{1/3}}\right). 
\end{aligned}
\end{gather}
Summing the inequalities in \eqref{new_eq:25} together, we obtain \eqref{new_eq:24}. 

\begin{enumerate}
\setcounter{enumi}{1}
\item $\alpha_0 \le L^{1/3}/(\pi \gamma^{1/3})$: we simply choose $\sigma = \dfrac{1}{\alpha_0}$. Then 
\begin{align*}
&48  L^2 d \sigma^2 \le \frac{48 L^2 d}{\alpha_0^2}, \ \  \frac{ 3 \gamma^2{d}}{\pi^2\sigma^4} \exp\left(-{4\pi^2 \alpha_0^2 \sigma^2 }\right)   \le \gamma^2 d \alpha_0^4 \le \frac{L^2 d}{\alpha_0^2} , 
\end{align*}
\end{enumerate}
directly implying \eqref{new_eq:24b}. The proof is concluded.
%
\end{proof}


\subsection{Diminishing noise}
\label{sec:dimin_noise}
In the previous subsection, the convergence analysis for wave-like noise models has been conducted. We tightened the neighborhood of convergence and proved that its size scales linearly with the noise wavelength with an appropriate choice of smoothing radius. 
%
%
Here, under the scenario that the noise decays to $0$ as $\bm x \to \bm x^*$, we show that the linear convergence of \eqref{alg:GD} to the \textit{exact} global minimum can be guaranteed. This result can be applied independently to general noisy functions with diminishing noise. It can also serve as a guideline for functions with oscillating, high-frequency noise, for example, when the perturbation to objective function can be represented by wave-like noise globally but diminishing noise near the global minimum. {By} progressively decreasing {the} smoothing radius after the optimizer reaches the neighborhood of convergence, {one} can break that limit. 

\begin{lemma}
\label{cor:bound_Bt}
Let $F$ be a function defined on $\mathbb{R}^d$, which can be represented as $F=\phi + \epsilon$, where $\phi$ satisfies Assumption \ref{assump:lipschitz_cont}, and $\bm x^*$ is the global minimum of $\phi$. Assume $|\epsilon(\bm x)| \le \beta \|\bm x- \bm x^*\|^2,\, \forall \bm x\in \mathbb{R}^d$, then we have 
\begin{align}
    \label{new_eq:19}
    B_t \le\,  \frac{ C {\pi} (M\,!)^2 d}{4^M ((2M)\,!)^2} \sigma_t^{4M-2} + 48  L^2 d \sigma_t^2 +   {\frac{12\beta^2 d}{\pi}} \left(4 \sigma_t^2 + \dfrac{r_t^4}{\sigma_t^2} \right),\, \forall t > 0. 
\end{align}
\end{lemma}
\begin{proof}
This lemma is a direct consequence of Proposition \ref{prop:dim_noise} and Lemma \ref{Lemma:bound_Bt}.
\end{proof}

The main result of this subsection is as follow. 
\begin{theorem}[Convergence on quadratically diminishing noise models]
\label{thm:main3}
Let $F = \phi + \epsilon$, where $\phi$ satisfies Assumptions \ref{assump:lipschitz_cont}--\ref{assump:strongly_convex} and $\bm x^*$ is the global minimum of $\phi$. Assume $|\epsilon(\bm x)| \le \beta \|\bm x- \bm x^*\|^2,\, \forall \bm x\in \mathbb{R}^d$, with $\beta$ satisfying
\begin{align}
\label{cond:beta}
 {  \beta \sqrt{2L^2\pi + \beta^2} } < \frac{\pi }{32d}\cdot\frac{8\tau^2 L}{48 L + 3\tau}. 
\end{align}
Let $\{\bm x_t\}_{t\ge 0}$ be generated by Algorithm \eqref{alg:GD} with $M$ satisfying $\left(\dfrac{16M}{e}\right)^M\! \ge \dfrac{\pi}{48\sqrt{2}L^2}$, $\lambda = {1}/({16L})$ and
\begin{align}
 \sigma_t = \dfrac{\sqrt{\beta}}{\sqrt[4]{8L^2\pi + 4\beta^2}}   \left[(1-\frac{\tau}{32L})   + (\frac{6}{\tau L}  + \frac{3}{8L^2}) \frac{ d \beta \sqrt{2L^2\pi + \beta^2} }{\pi}\right]^{t/2}  \tilde{r}_0,\, 
 \label{new_eq:22}
\end{align}
with some $\tilde{r}_0\ge \|\bm x_0 - \bm x^*\|$. Then $\{\bm x_t\}_{t\ge 0}$ exponentially convergence to $\bm x^*$: 
\begin{align}
\label{new_eq:23}
\| \bm x_t - \bm x^*\|^2 \le  \left[(1-\frac{\tau}{32L})   + (\frac{6}{\tau L}  + \frac{3}{8L^2}) \frac{ d \beta \sqrt{2L^2\pi + \beta^2} }{\pi}\right]^{t}  \tilde{r}_0^2,\, \forall t \ge 0. 
\end{align} 
\end{theorem}

\begin{proof}
For any sequence $\{\sigma_t\}_{t\ge 0}$ of the smoothing radii in Algorithm \eqref{alg:GD}, from Lemmas \ref{lemma:bound_rt} and \ref{cor:bound_Bt}, there holds
\begin{gather}
\begin{aligned}
  r_{t+1}^2  &{\le} \,      
 (1 - (\lambda \tau -8\lambda^2 \tau L)) r_{t}^2
 +  \left(\frac{2\lambda}{\tau}  +  2 \lambda^2 \right) B_t. 
\\
 & = (1-\frac{\tau}{32L}) r_t^2 + (\frac{1}{8\tau L} + \frac{1}{128L^2})
 \\
&\qquad\qquad \cdot\left( \frac{ C{\pi} (M\,!)^2 d}{4^M ((2M)\,!)^2} \sigma_t^{4M-2} + 48  L^2 d \sigma_t^2 +   {\frac{12\beta^2 d}{\pi}} \left(4 \sigma_t^2 + \dfrac{r_t^4}{\sigma_t^2} \right)\right) 
 \\
& \le   (1-\frac{\tau}{32L}) r_t^2  + (\frac{1}{8\tau L}  + \frac{1}{128L^2})
 \left((96 L^2 d +   {\frac{48\beta^2 d}{\pi}})  \sigma_t^2 +   {\frac{12\beta^2 d}{\pi}} \cdot \dfrac{r_t^4}{\sigma_t^2} \right). 
 \end{aligned}
 \label{new_eq:20}
 \end{gather}
Note that in the last inequality, we also used the estimate $\dfrac{ {\pi} (M\,!)^2 d}{4^M ((2M)\,!)^2} \le 48L^2d$, derived from the Stirling's formula $M! \simeq \sqrt{2\pi M} ({M}/{e})^M$, \cite[Section 6.1.38]{Handbook}, and the condition on $M$. Let $\{\tilde{r}_t\}_{t\ge 0}$ be a sequence satisfying $r_0 \le \tilde{r}_0$ and defined recursively as
 \begin{align*}
 \tilde{r}_{t+1}^2  & :=   (1-\frac{\tau}{32L}) \tilde{r}_t^2  + (\frac{1}{8\tau L}  + \frac{1}{128L^2}) \left((96 L^2 d +   {\frac{48\beta^2 d}{\pi}})  \sigma_t^2 +   {\frac{12\beta^2 d}{\pi}} \cdot \dfrac{\tilde{r}_t^4}{\sigma_t^2} \right). 
 \end{align*}
 It is easy to see $r_t \le \tilde{r}_t,\, \forall t>0$. Choosing $\sigma_t$ such that
 $
 \sigma^2_t = \dfrac{\beta}{\sqrt{8L^2\pi + 4\beta^2}} \cdot \tilde{r}_t^2,
 $
 we will prove 
 \begin{align}
 \tilde{r}^2_t = \left[(1-\frac{\tau}{32L})   + (\frac{6}{\tau L}  + \frac{3}{8L^2}) \frac{ d \beta \sqrt{2L^2\pi + \beta^2} }{\pi}\right]^{t}  \tilde{r}_0^2,\, \forall t\ge 0. 
 \label{new_eq:21}
 \end{align}
 First, \eqref{new_eq:21} is trivial for $t=0$. Assume \eqref{new_eq:21} holds for some $t\ge 0$, we have from the definition of $\tilde{r}_{t+1}$ and $\sigma_t$ 
  \begin{align*}
 \tilde{r}_{t+1}^2  &= (1-\frac{\tau}{32L}) \tilde{r}_t^2  + (\frac{1}{8\tau L}  + \frac{1}{128L^2}) \frac{48 d \beta \sqrt{2L^2\pi + \beta^2} }{\pi}  \tilde{r}_t^2
 \\
  &=    \left[(1-\frac{\tau}{32L})   + (\frac{6}{\tau L}  + \frac{3}{8L^2}) \frac{ d \beta \sqrt{2L^2\pi + \beta^2} }{\pi}\right]  \tilde{r}_t^2 
  \\
  &=     \left[(1-\frac{\tau}{32L})   + (\frac{6}{\tau L}  + \frac{3}{8L^2}) \frac{ d \beta \sqrt{2L^2\pi + \beta^2} }{\pi}\right]^{t+1}  \tilde{r}_0^2, 
 \end{align*}
{where the last equation follows from an inductive process.}
 Since $r_t \le \tilde{r}_t$, this implies the estimate \eqref{new_eq:23} of $\|\bm x_t - \bm x^*\|$. Under condition \eqref{cond:beta}, we have $\|\bm x_t - \bm x^*\|$ decays exponentially as $t\to \infty$. The adaptive rule of $\sigma_t$ can be drawn directly from the relation between $\sigma_t$ and $\tilde{r}_t$. 
 \end{proof}

 We conclude this section by noting that the condition \eqref{cond:beta} can be simplified as 
 $
 ({\beta}/{L}) \sqrt{2\pi + ({\beta}/{L})^2} < \frac{\pi}{204 d}, 
 $
 thus $\beta \lesssim {L}/{\sqrt{d}}$. 

\section{Numerical experiments}
\label{sec:experiment}
We perform several numerical experiments to illustrate and confirm our theoretical results, including: 
\begin{itemize}[leftmargin=10pt]
\item The linear convergence of Algorithm \eqref{alg:GD} to a neighborhood of convergence,
\item The correlation between the optimal range of smoothing radius and the period or maximum wavelength of the noise,
\item The convergence to {the} exact global minimum by exponentially {decreasing} the smoothing radius in case of a diminishing noise.
\end{itemize}
For this purpose, we consider simple test problems where we seek to minimize a smooth, unimodal function $\phi$ from its noisy estimate 
$$
F(\bm x) = \phi (\bm x)+\epsilon (\bm x),
$$
with $\epsilon$ being an oscillating function that falls into one class of noise models defined in Section \ref{sec:statement}. More extensive evaluations on {the practical performances of the DGS method on}  high-dimensional benchmark functions and numerical comparisons of DGS method with other baselines can be found in \cite{DBLP:conf/uai/ZhangTLZ21, AdaDGS_20}.

In the first test, we consider the $5$D synthetic functions
\begin{align}
\label{exp:periodic}
\phi (\bm x) = \left(\sum_{i=1}^{5} | x_i^{2 + i}| \right)^{1/2}\,\ \text{ and } \ \epsilon(\bm x) = \sum_{i=1}^{5} \sin(2\pi \alpha x_i),
\end{align}
defined on $[-20,20]^5$. Here, $\epsilon$ is a wave-like periodic function on each standard orthonormal direction with period $1/\alpha$. We perform the test for $\alpha = 1,\, 1/2,\, 1/4$, so the periods of the noise functions are $1,\, 2,\, 4$ correspondingly. Algorithm \eqref{alg:GD} is run with $\lambda = 0.001$ and different values of $\sigma$, ranging from $\sigma = 0.01*(1/\alpha)$ to $50*(1/\alpha)$. For each configuration, we perform 20 trials, each of which has a random initial state, and report the results averaged over those trials. 
Maximum $50000$ iterations are allowed for each run. Figure \ref{fig:noise_period2} evaluates the quality of DGS gradients, measured by their cosine similarity to $\nabla \phi$ along the optimization path. In all cases, we observe that both small and big $\sigma$ lead to a degradation of the accuracy of the DGS gradients, and the optimal range of $\sigma$ is near $1/\alpha$, the period of the noise function. Figure \ref{fig:noise_period1} shows the evolution of the distance between the approximated solution and global minimum $\|\bm x - \bm x^*\|$. We see that Algorithm \eqref{alg:GD} enjoys an exponential convergence rate before stagnating at some spurious local minimum. Also, the accuracy of DGS gradients is a reliable indicator for the convergence performance of \eqref{alg:GD}. In all three cases, starting from the smallest considered value of $\sigma$, the 
neighborhood of convergence shrinks when we increase the smoothing radius. It {achieves} the minimal size when $\sigma$ is $1/\alpha$, after which the trend is reversed. These results are consistent with Theorem \ref{thm:main}, with an exception that the size of the neighborhood of convergence is generally smaller than those guaranteed in the theory, {since the latter is based on the worst-case scenario}.

\begin{figure}[h]
\vspace{-.1cm}
\includegraphics[width = 1.5in]{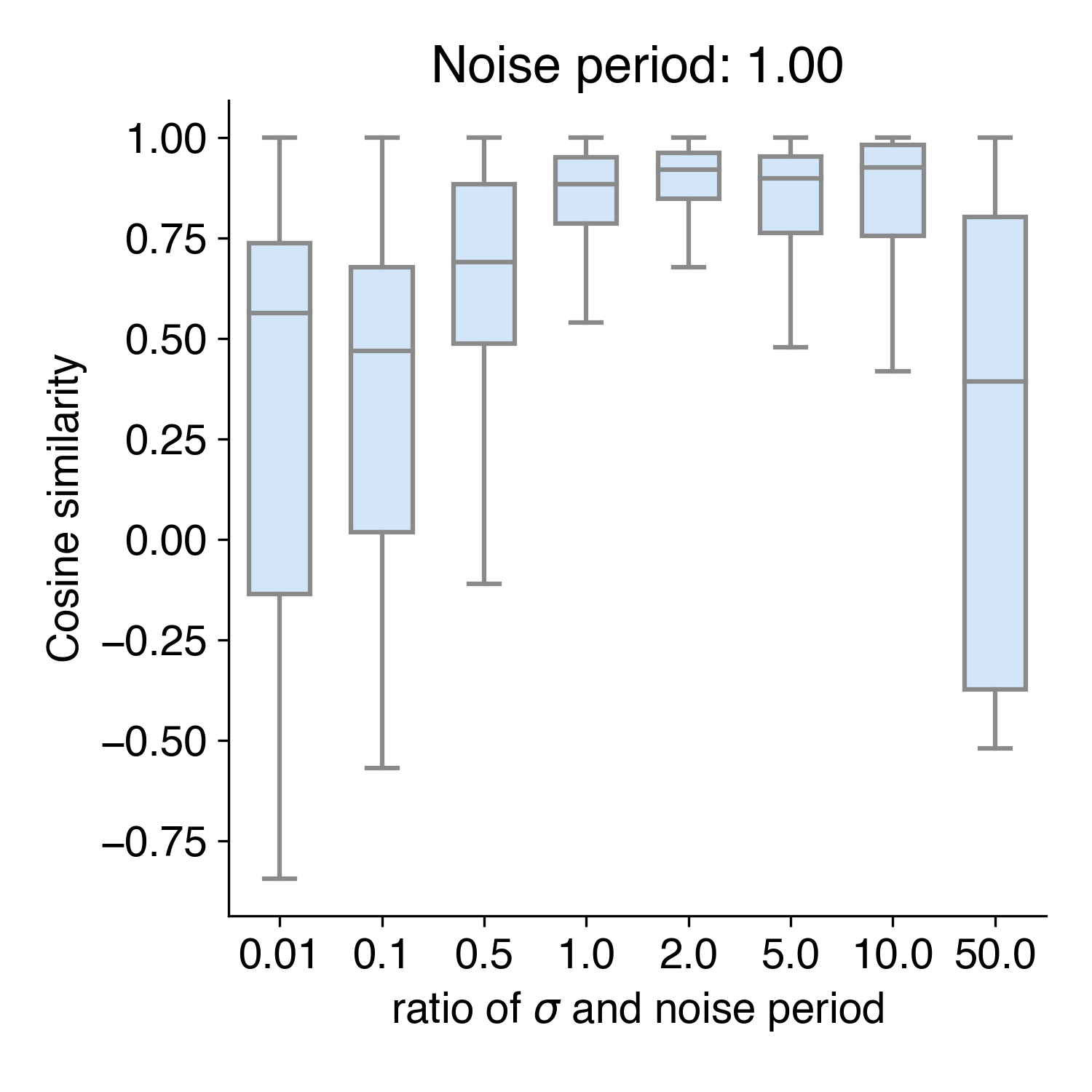}
\includegraphics[width = 1.5in ]{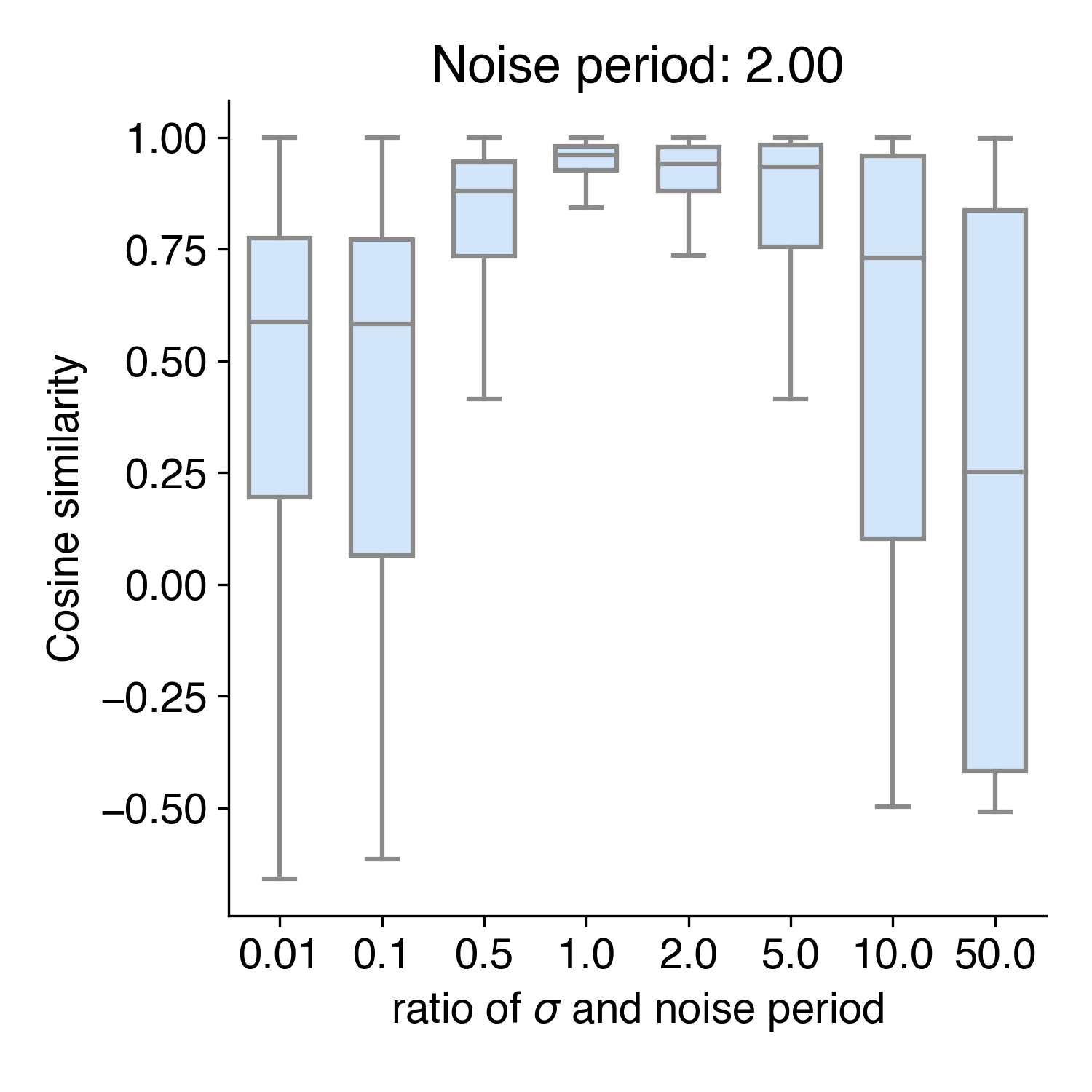}
\includegraphics[width = 1.5in ]{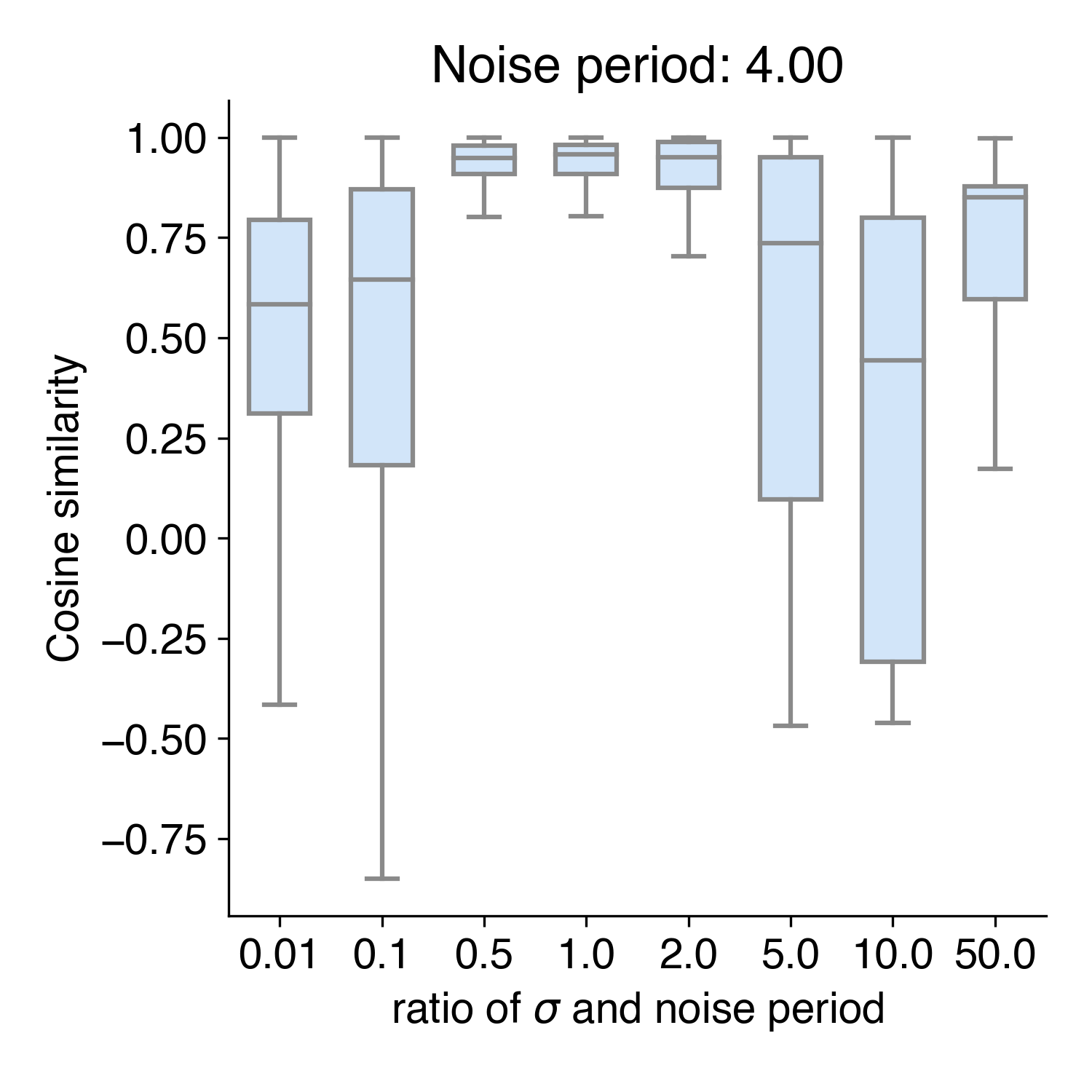}
\vspace{-.1in}
\caption{The cosine similarity between DGS gradients and $\nabla \phi$ along the optimization path in minimizing $\phi$ with the periodic noise defined in \eqref{exp:periodic}. From left to right, the noise periods are $1$, $2$ and $4$ respectively. For each case, we run Algorithm \eqref{alg:GD} with multiple different values of smoothing radius.}
\label{fig:noise_period2}
\vspace{-.2in}
\end{figure}

\begin{figure}[h]
\vspace{-.1cm}
\includegraphics[width = 1.5in]{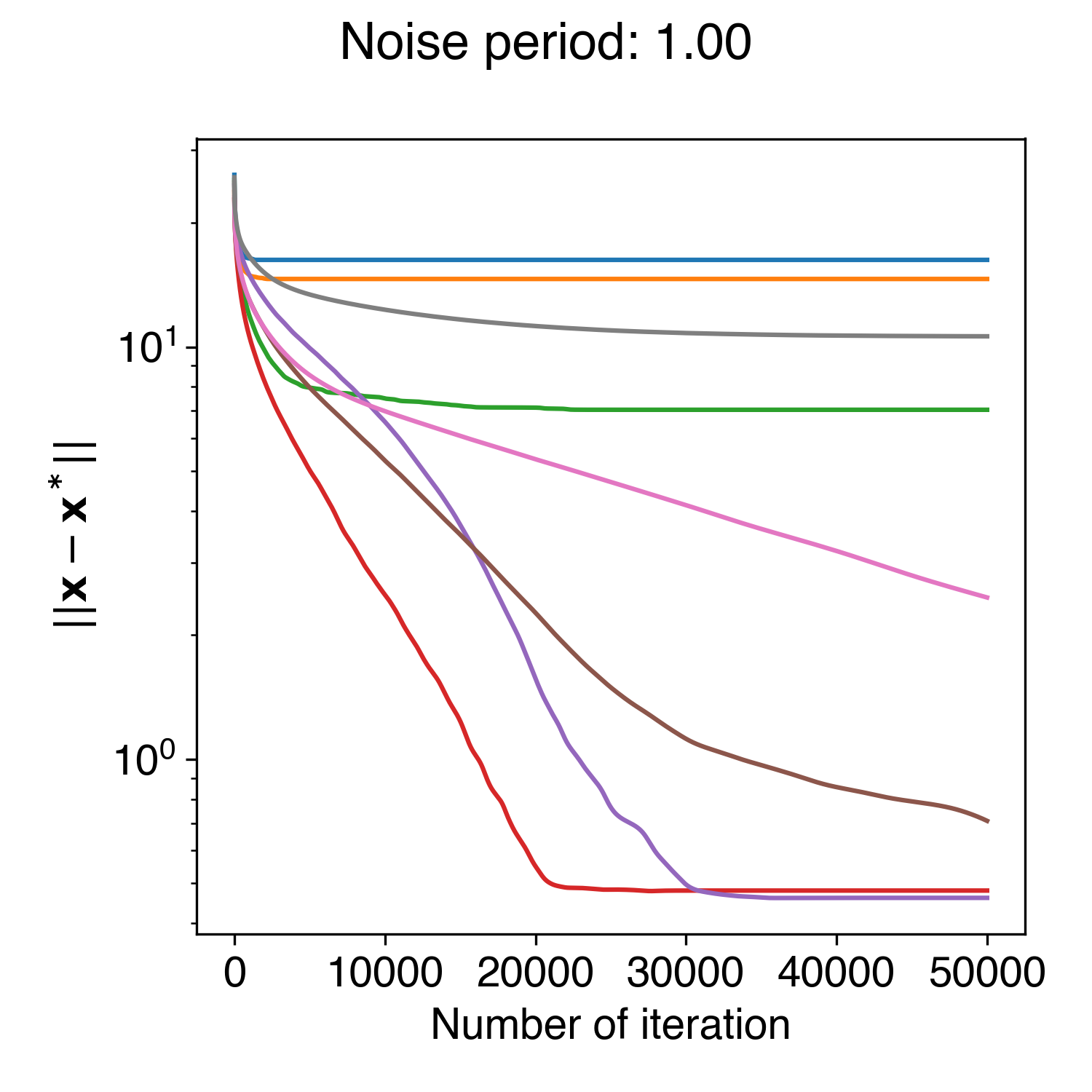}
\includegraphics[width = 1.5in ]{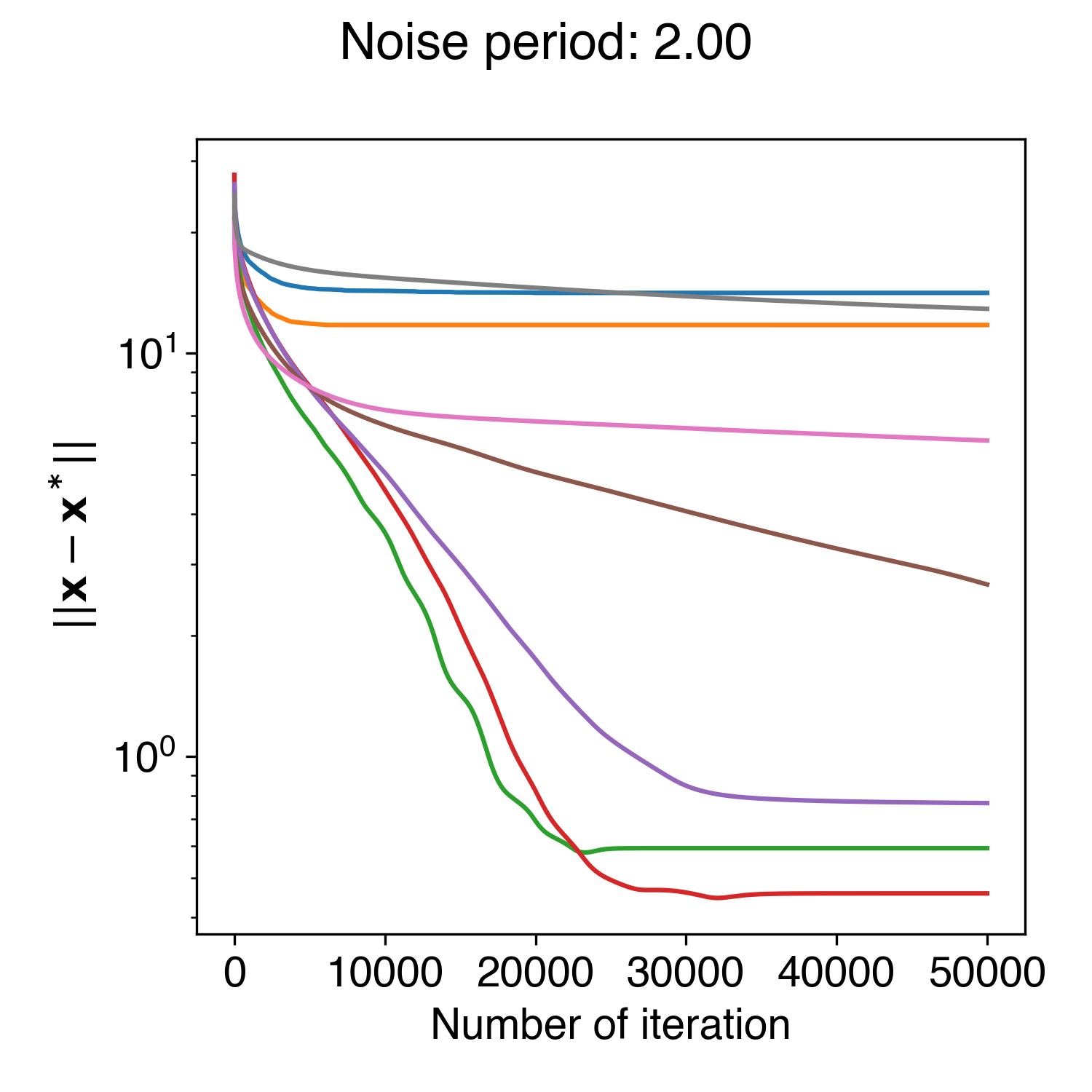}
\includegraphics[width = 1.9in ]{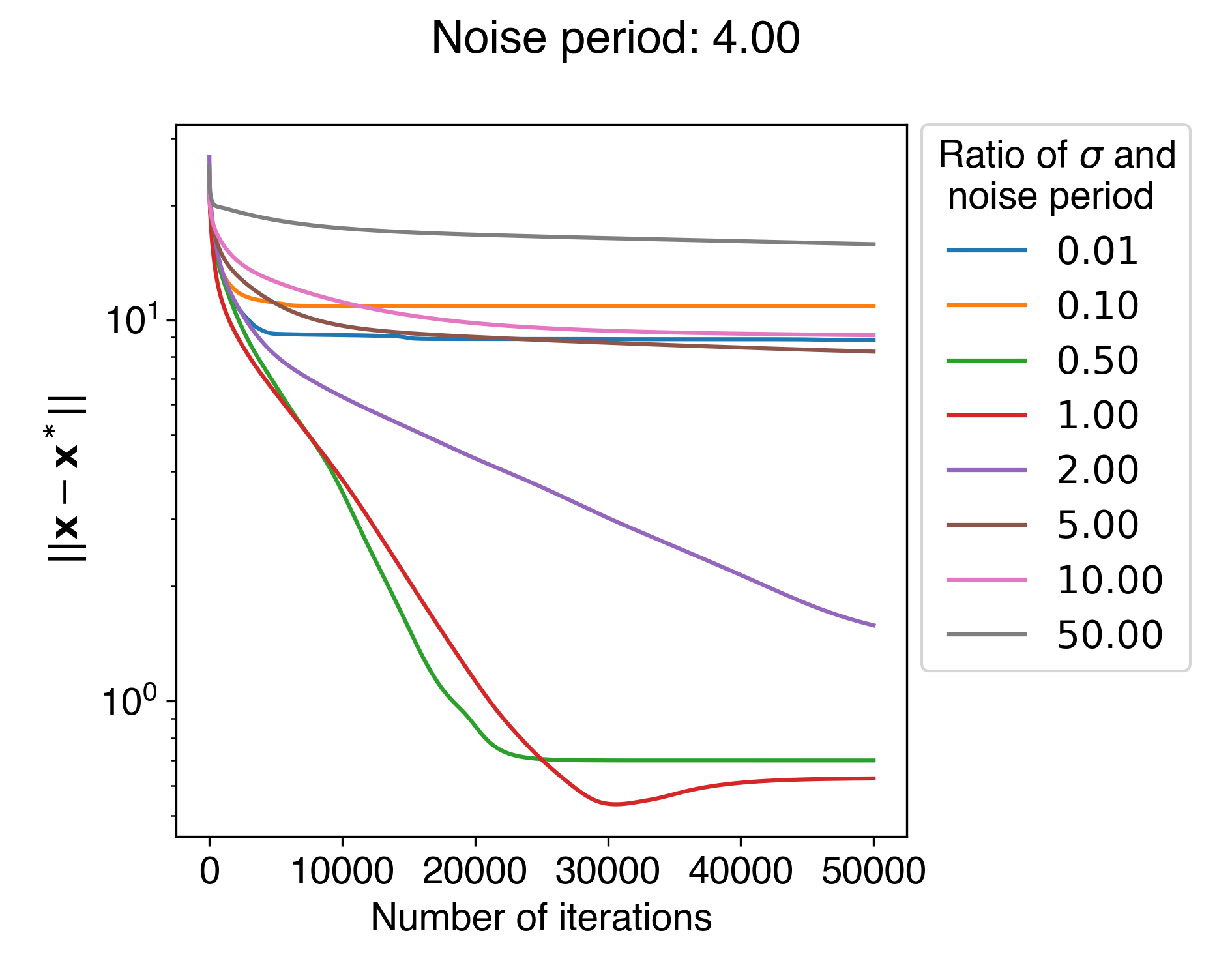}
\vspace{-.1in}
\caption{The convergence of Algorithm \eqref{alg:GD} in minimizing $\phi$ defined in \eqref{exp:periodic}. It can be seen that the optimal value of smoothing radius is equal to the period of the noise. }
\label{fig:noise_period1}
\vspace{-.2in}
\end{figure}

In the second test, to represent the high-frequency bandlimited noise model, we consider the synthetic functions
\begin{align}
\label{exp:BL}
\phi (\bm x) = \left(\sum_{i=1}^{5}  | x_i^{2 + i}| \right)^{1/2}\,\ \text{ and } \ \epsilon(\bm x) = \frac{1}{20}\sum_{i=1}^{5} \sum_{j=1}^{20} \sin(2\pi \alpha_{i,j} x_i),
\end{align}
on $[-20,20]^5$, where $\{\alpha_{i,j}\}$ is defined such that $1/\alpha_{i,j}$ is uniformly sampled from $[0,1/\alpha_0]$. Thus, on each standard orthonormal direction, $\epsilon$ is an aggregation of $20$ periodic functions, with periods uniformly distributed on $[0,1/\alpha_0]$. With this setup, the maximum wavelength of $\epsilon$ is approximately $1/\alpha_0$. We perform the test for $\alpha_0 = 1,\, 1/2,\, 1/4$. 
Similar to previous test, Algorithm \eqref{alg:GD} is run with $\lambda = 0.001$ and different values of $\sigma$, ranging from $\sigma = 0.01*(1/\alpha_0)$ to $50*(1/\alpha_0)$. For each configuration, we perform 20 trials with different random initial state, and
report the results averaged over those trials. Maximum $70000$ iterations are allowed
for each run. Figure \ref{fig:noise_period4} below evaluates the quality of DGS gradients, measured by their cosine similarity to $\nabla \phi$ along the optimization path. We observe that the DGS gradient most accurately approximates $\nabla \phi$ with $\sigma$ being around $1/\alpha_0$. In Figure \ref{fig:noise_period3}, we see the exponential convergence of Algorithm \eqref{alg:GD} to some neighborhood of the global minimum. The size of this neighborhood is minimal when $\sigma$ is near $1/\alpha_0$, and increases when the smoothing radius moves to either left or right of this optimal point. These results are consistent with Theorem \ref{thm:main2}.

\begin{figure}[h]
\vspace{-.1cm}
\includegraphics[width = 1.5in]{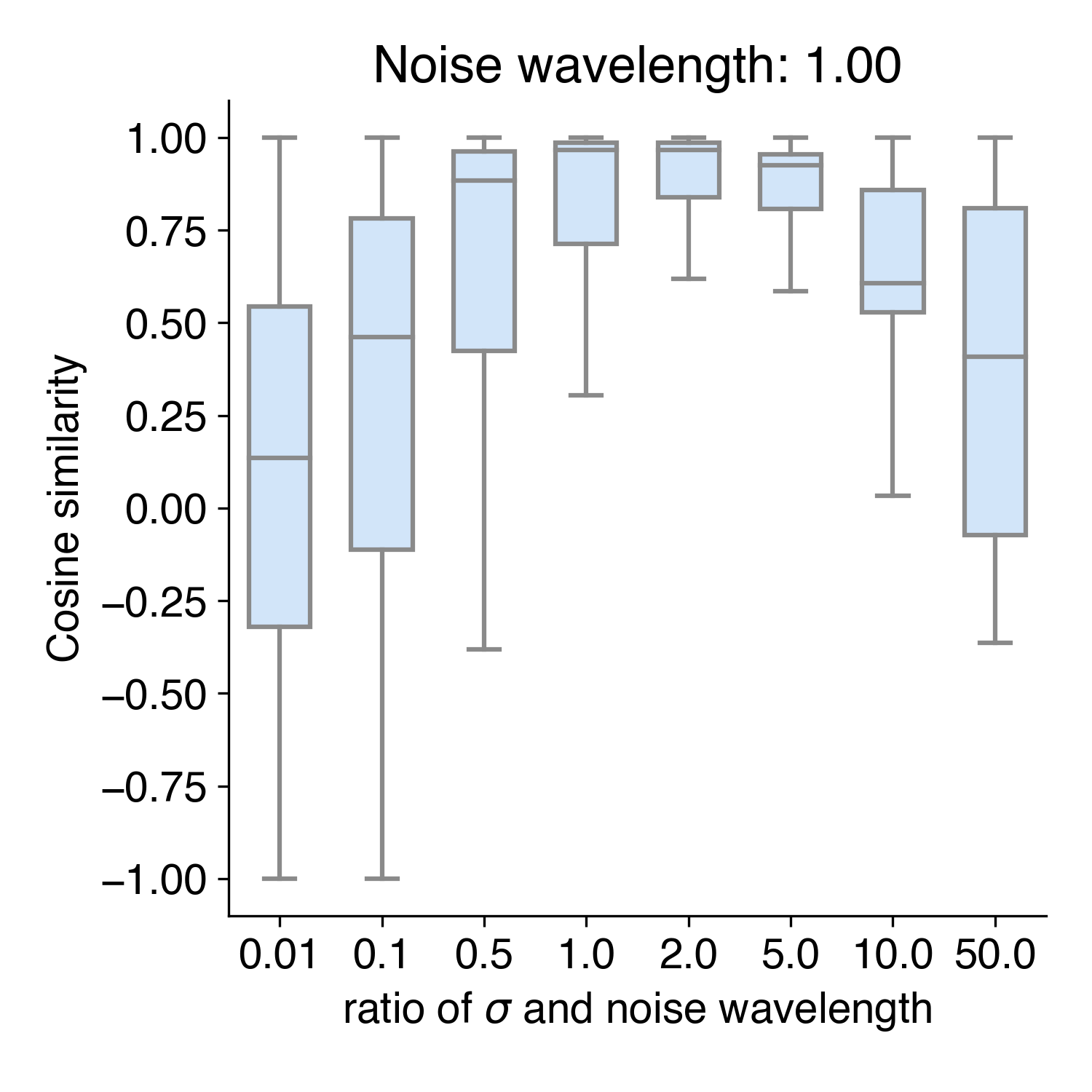}
\includegraphics[width = 1.5in ]{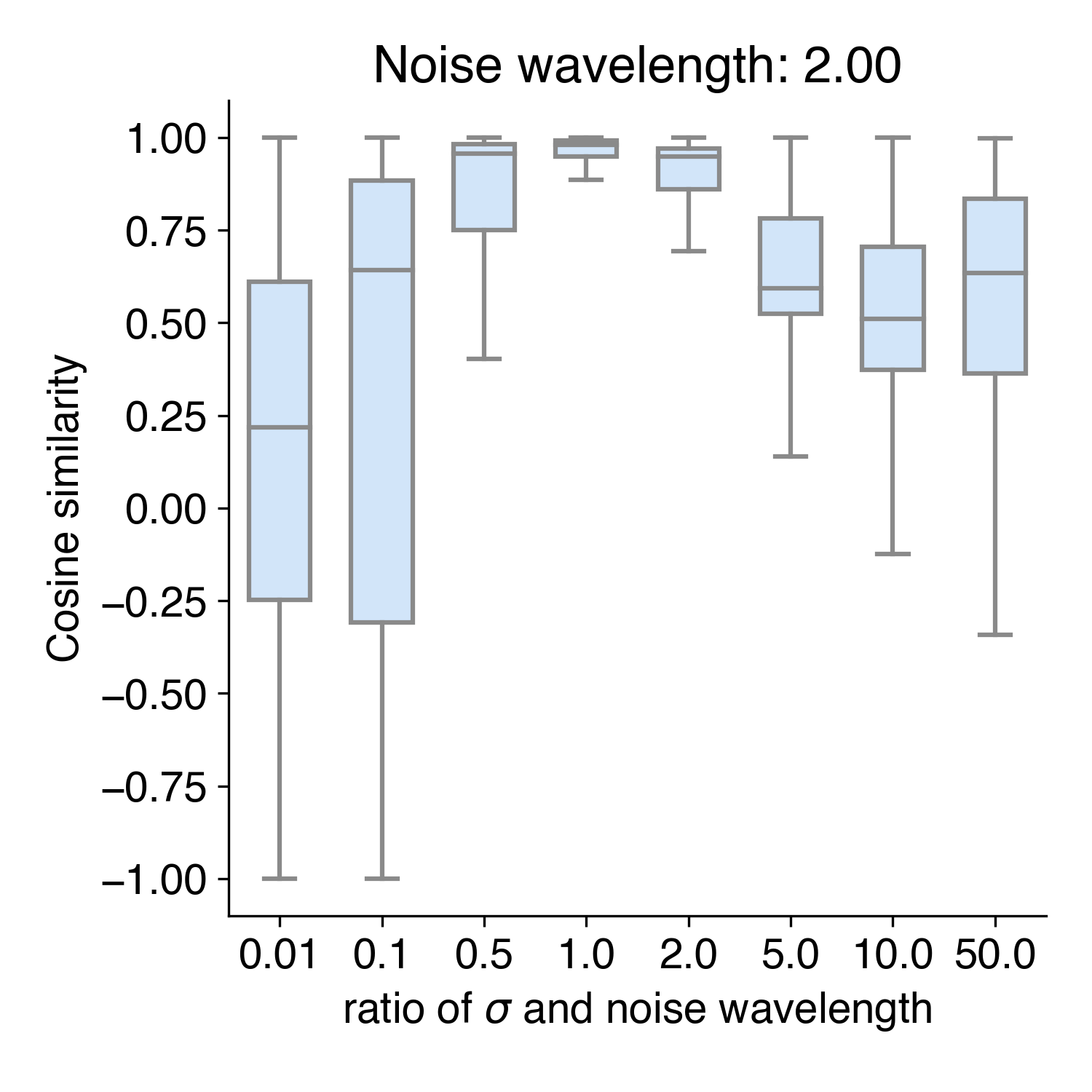}
\includegraphics[width = 1.5in ]{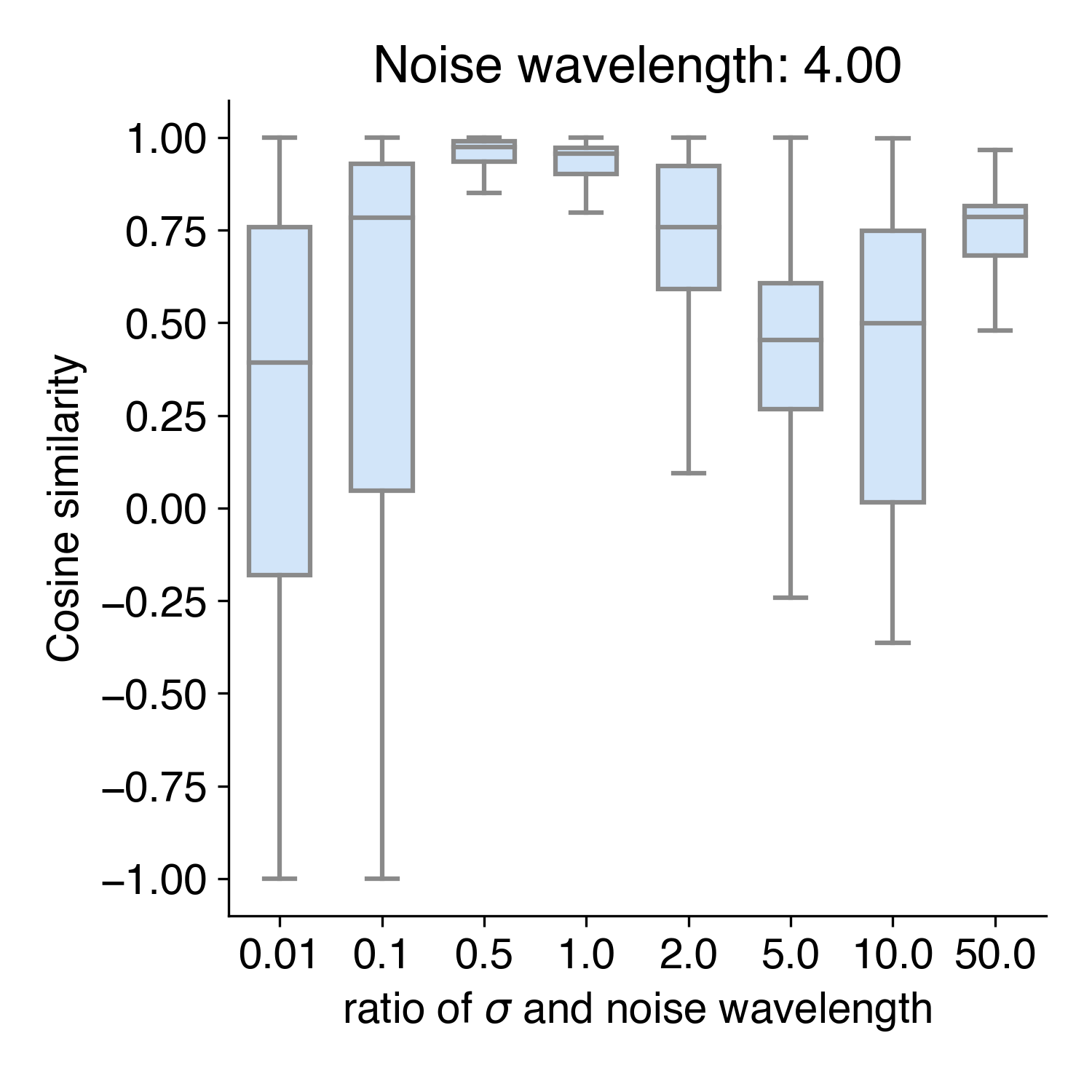}
\vspace{-.1in}
\caption{The cosine similarity between DGS gradients and $\nabla \phi$ along the optimization path in minimizing $\phi$ with the high-frequency bandlimited noise defined in \eqref{exp:BL}.}
\label{fig:noise_period4}
\vspace{-.2in}
\end{figure}

\begin{figure}[h]
\vspace{-.1cm}
\includegraphics[width = 1.5in]{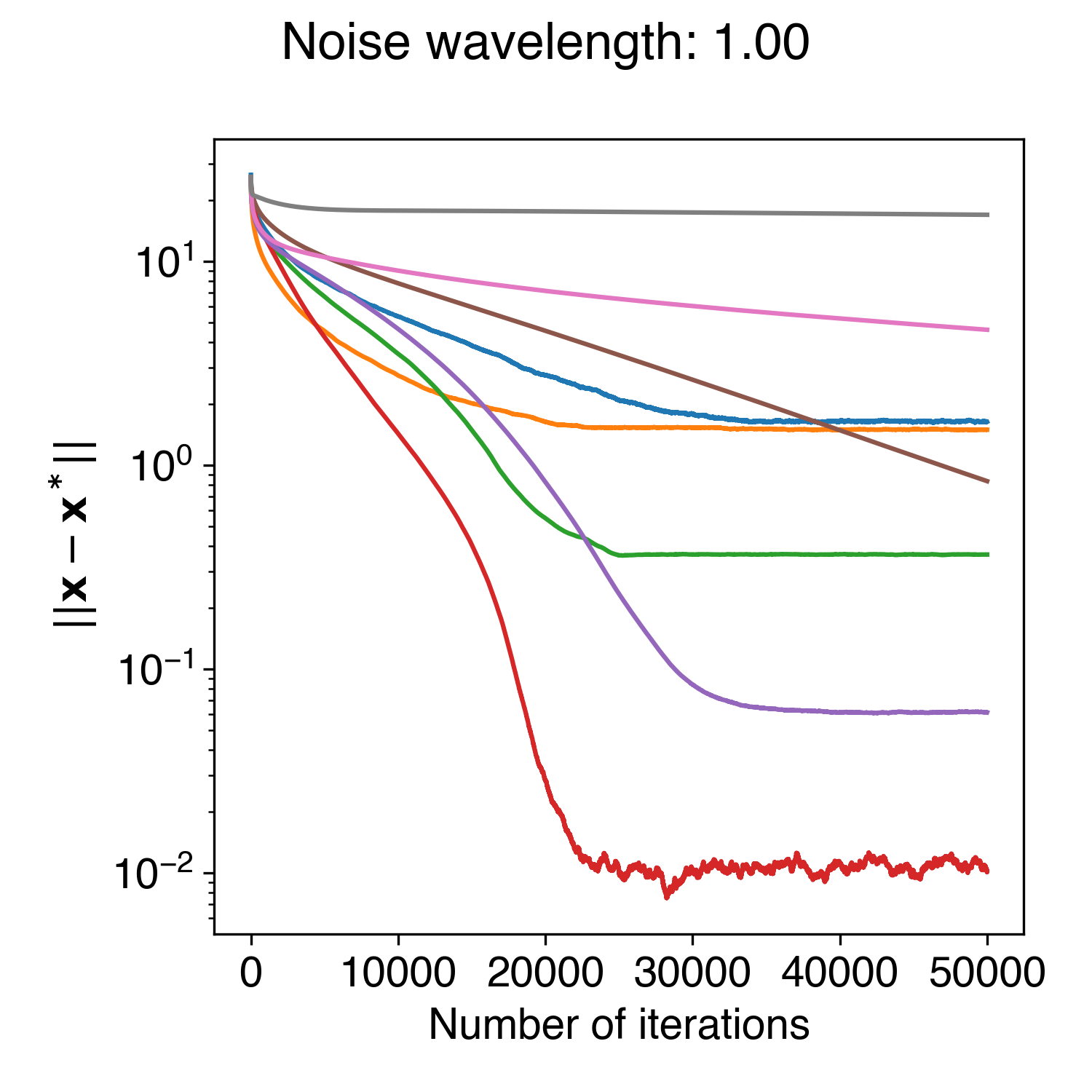}
\includegraphics[width = 1.5in ]{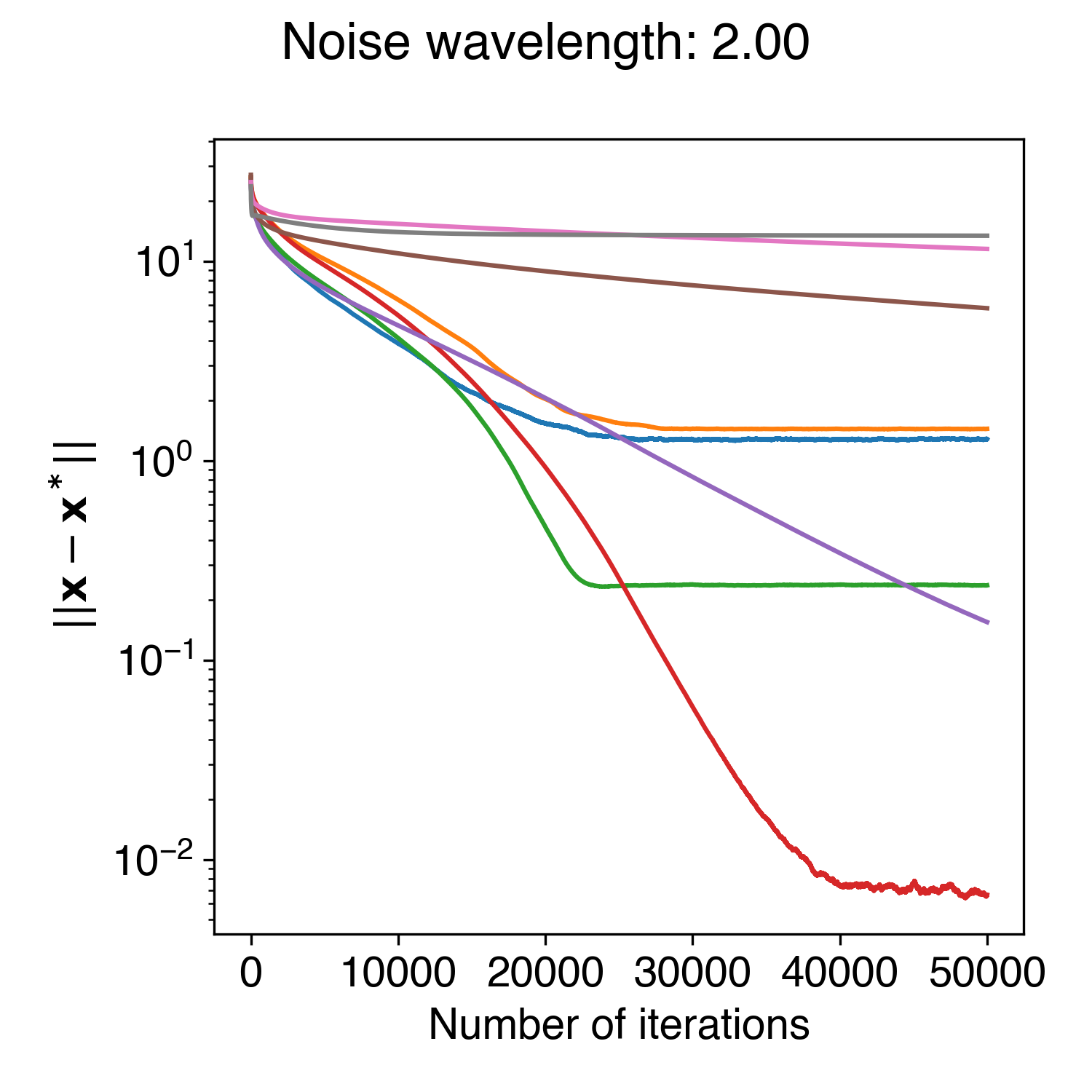}
\includegraphics[width = 1.9in ]{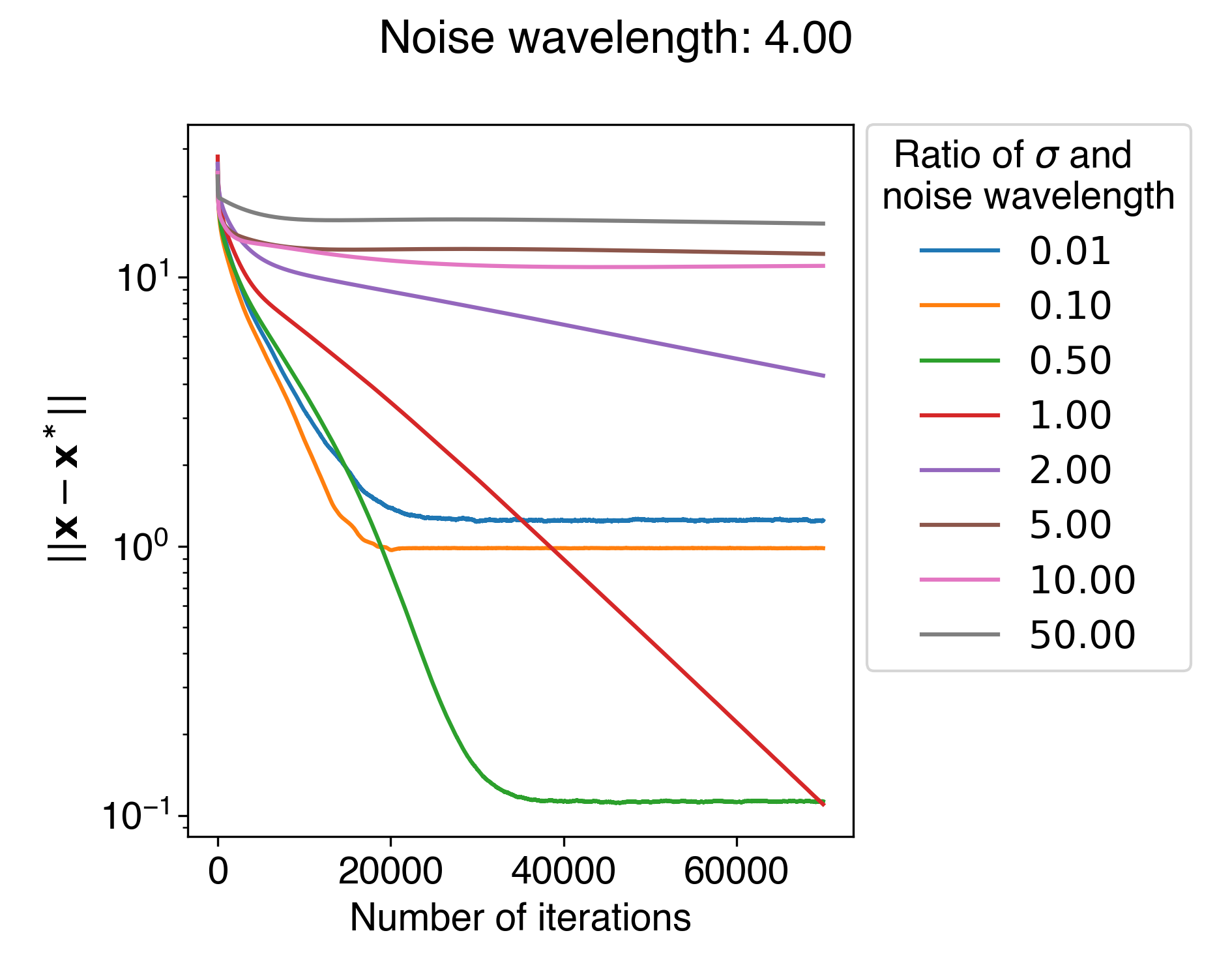}
\vspace{-.1in}
\caption{The convergence of Algorithm \eqref{alg:GD} in minimizing $\phi$ defined in \eqref{exp:BL}. It can be seen that the optimal value of smoothing radius is around to the maximum wavelength of the noise.}
\label{fig:noise_period3}
\vspace{-.2in}
\end{figure}

In the third test, to illustrate that Algorithm \eqref{alg:GD} can converge to the exact global minimum in case of a diminishing noise, we consider the following functions 
\begin{align}
\label{exp:ND}
\phi (\bm x) = \sum_{i=1}^{5} x_i^{2 } \,\ \text{ and } \ \epsilon(\bm x) = \sum_{i=1}^{5} x_i^2 \sin(2\pi  x_i), 
\end{align}
defined on $[-5,5]^5$. Here, $\epsilon$ satisfies $|\epsilon(\bm x)| \le \|\bm x - \bm x^*\|^2$, and the wavelength of $\epsilon$ is $1$. We perform Algorithm \eqref{alg:GD} with $\lambda = 0.005$ and allow maximum $20000$ iterations for each run. To test the effectiveness of the exponential decreasing updating rule for $\sigma$, we consider a two-phase update schedule: first, $\sigma$ will be fixed at a base value $\sigma_0$ for the first $5000$ iterations; then, $\sigma$ is reduced after each iteration with a contraction rate $0.999$. Different base values of $\sigma$ are tested, ranging from $\sigma_0 = 0.01*(\text{wavelength})$ to $50*(\text{wavelength})$. For
each configuration, we perform 20 trials with random initial state and report the average results. Figure \ref{fig:noise_period5} presents the convergence results of Algorithm \eqref{alg:GD}. Observe that regardless of the base value $\sigma_0$, the algorithm becomes stagnating after a few hundred iterations. After $5000$ iterations when the exponential decaying rule is applied, the optimizer is able to make progress again and we see a significant deviation in its performance with respect to $\sigma_0$. For very small $\sigma_0$, there is no change in the second phase, because reducing the already too small smoothing radius is not helpful. For very big $\sigma_0$, we observe the linear convergence to the global minimum, but it is delayed, seemingly until the smoothing radius is reduced to an appropriate range. The optimal value of $\sigma_0$ that features the fastest convergence, soon after the dynamic update rule is activated, are those near the wavelength value. This experiment does not only confirm the exponential convergence of algorithm \eqref{alg:GD} in case of diminishing noise, but also emphasizes the strong correlation between optimal smoothing radius and the noise wavelength.   

\begin{figure}[h]
\vspace{-.1cm}
\includegraphics[width = 2.in]{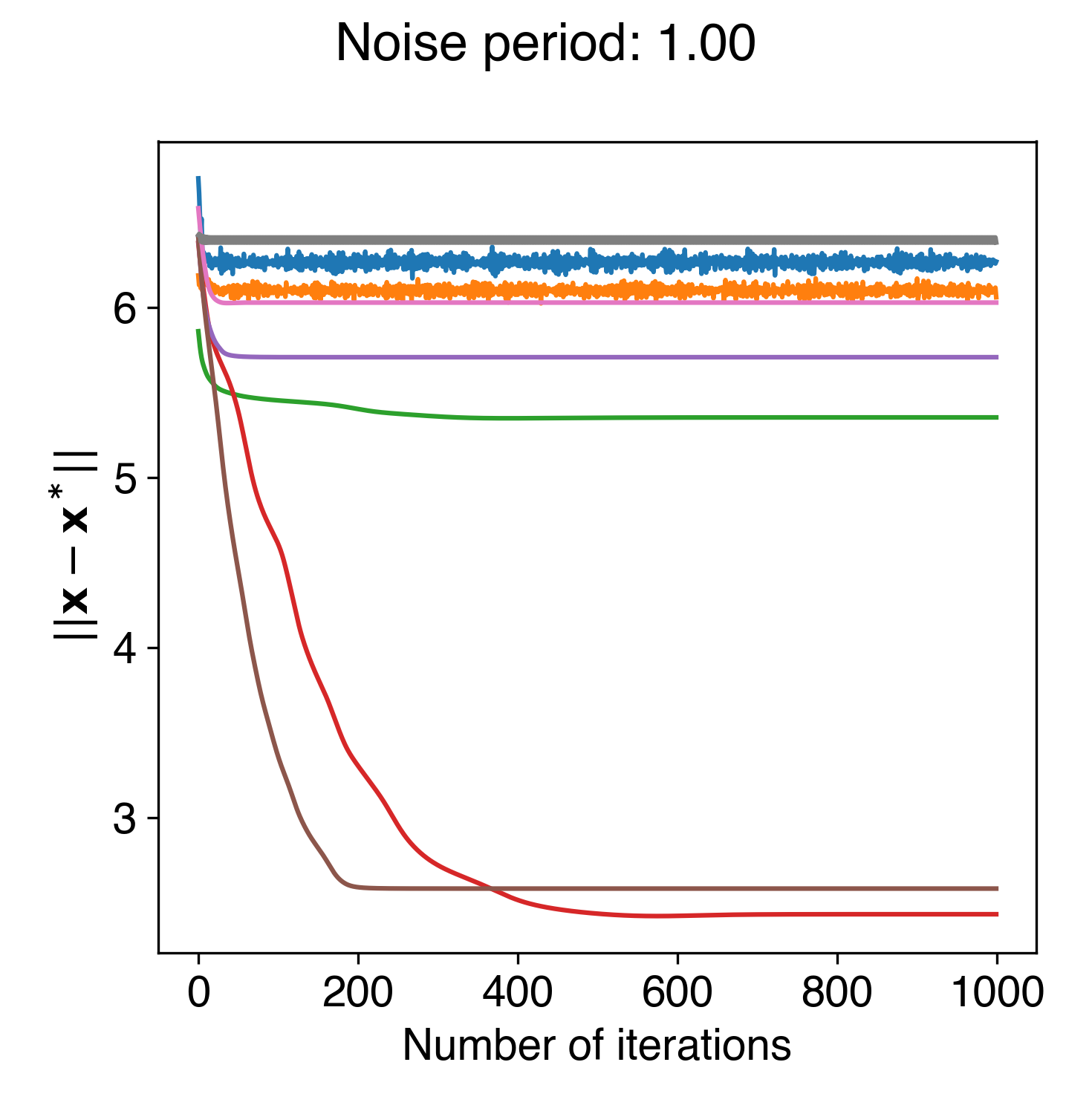}
\includegraphics[width = 2.7in ]{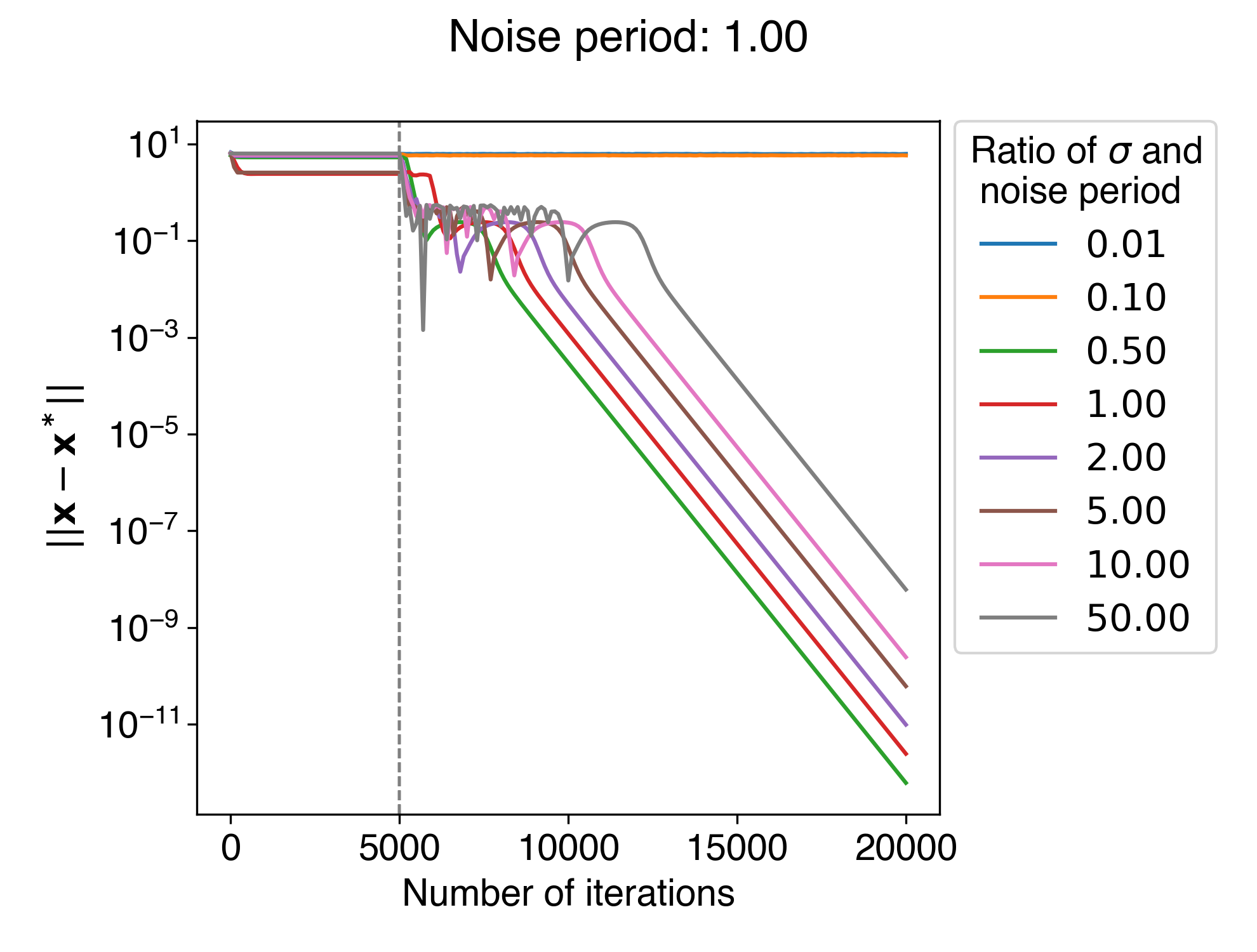}
\vspace{-.1in}
\caption{The convergence of Algorithm \eqref{alg:GD} in minimizing $\phi$ defined in \eqref{exp:ND} with the diminishing noise. {Left:} the algorithm becomes
stagnating after a few hundred iterations when fixed smoothing radius is used. Right: the linear convergence of \eqref{alg:GD} after the smoothing radius is exponentially decreased.}
\label{fig:noise_period5}
\vspace{-.2in}
\end{figure}

\section{Conclusion}
\label{sec:conclusion}

In this paper, we present a convergence theory for a gradient descent scheme that uses the DGS gradient to guide the nonlocal search on noisy, multimodal landscapes. Under the scenario that the objective function is composed of a convex function, perturbed by a highly oscillating, deterministic noise, we prove the linear convergence of the method to neighborhood of the solution, which can be significantly tightened with an appropriate choice of smoothing radius. Unlike existing theories on optimization under the presence of noise, which were based on the noise magnitude, our analysis takes advantage the frequency/wavelength information of the noise and shows a strong correlation between the optimal values of smoothing radius, the size of the neighborhood of convergence and the noise wavelength. We believe that these results are just cursory and more can be drawn from studying the noisy functions from this aspect.   

As our main interest in this paper is mathematical theory, we only consider the DGS algorithm in its simplest form, and skip many features that are important for the efficiency and practicality of the approach. For a more advanced version, which incorporates line search and self-tuning schedule for smoothing radius, we refer the interested reader to \cite{AdaDGS_20}. However, there are several questions arising from the present theory that can be beneficial to the algorithmic developments. First, an efficient method to estimate the frequency or wavelength (either local or global) of oscillating functions will be greatly useful. As we have seen, the optimal smoothing radius is informed by the wavelength. Second, the objective function $\phi$ is assumed here to be strongly convex. Under this scenario, gradient descent scheme is sufficient to guarantee the linear convergence. Weaker assumptions on $\phi$ can be considered, and with that, more involved gradient-based schemes for the optimization. Third, our analysis is based on the noise models that are characterized directionally. While our high-frequency bandlimited model can represent a wide class of fluctuating noise, bridging these and with more general noise models which are characterized globally is an important question that we want to study next. Finally, the DGS gradient is less computational demanding but potentially more noisy than the standard GS gradient in high dimension. We note that in practice, one can extend the idea further and develop a more flexible strategy which decompose the space into the direct sum of $k$ subspaces ($1\leq k \leq d$) and define a {\it{decomposition-based}} GS via independent GS within each subspace.
The DGS gradient corresponds to the special case of $k=d$ while the conventional GS can be recovered in the case $k=1$. Such decomposition can be designed a priori, or adaptively during the iterations. One can also introduce multilevel and/or hierarchical ideas to such a decomposition.  
These and other extensions will be studied separately in the future works.  

\section*{Acknowledgement}
This work was supported by the U.S. Department of Energy, Office of Science, Office of Advanced Scientific Computing Research, Applied Mathematics program, under the contract ERKJ387, and accomplished at Oak Ridge National Laboratory (ORNL). ORNL is operated by UT-Battelle, LLC., for the U.S. Department of Energy under Contract DE-AC05-00OR22725. The research of Qiang Du is supported in part by DE-SC0022317 and NSF DMS 2012562.


\end{document}